\documentclass[a4paper, 12pt]{article}
\usepackage{amsfonts}
\usepackage{amsmath}
\usepackage{amsthm}
\usepackage{amssymb}
\usepackage{mathrsfs}
\usepackage{bbm}
\usepackage{kotex}
\usepackage{graphicx}
\usepackage{color}
\usepackage{xcolor}

\definecolor{IRL}{rgb}{1.0, 0.5, 0.0}

\theoremstyle{theorem}
\newtheorem{theorem}{Theorem}

\newtheorem{definition}[theorem]{Definition}

\newtheorem{lemma}[theorem]{Lemma}

\newtheorem{proposition}[theorem]{Proposition}

\theoremstyle{remark}
\newtheorem{remark}[theorem]{Remark}

\newcommand{\h}{\mathsf{h}}
\newcommand{\B}{\mathcal{B}(\h)}

\newcommand{\unit}{\hbox{\rm 1\kern-2.8truept l}}
\newcommand{\T}{\mathcal{T}}

\newcommand{\HS}{\mathfrak{I}_2(\mathsf{h})}

\newcommand{\tr}{\hbox{\rm tr}}
\newcommand{\md}{\hbox{\rm d}}
\newcommand{\me}{\hbox{\rm e}}
\newcommand{\mi}{\mathrm{i}}

\newcommand{\XY}{X\kern-9trueptY}

\newcommand{\scalare}[2]{\left\langle #1 , #2 \right\rangle}
\newcommand{\rescalar}[2]{\Re \scalare{#1}{#2}}
\newcommand{\conj}[1]{\overline{#1}}
\newcommand{\scalar}[2]{\scalare{#1}{#2}}
\newcommand{\norm}[1]{\left\| #1 \right\|}
\newcommand{\ran}{\operatorname{Ran}}

\newcommand{\modulo}[1]{\left| #1 \right|}
\newcommand{\CZ}{\mathbf{C_Z}}
\newcommand{\identity}{\mathbbm{1}}

\newcommand{\outerp}[2]{\left| #1 \right\rangle \! \left\langle #2 \right|}
\begin{document}
\title{The Spectral Gap of a Gaussian Quantum Markovian Generator}
\author{F. Fagnola$^{(1)}$, D. Poletti$^{(2)}$, E. Sasso$^{(2)}$, V. Umanit\`a$^{(2)}$}

\maketitle

\begin{abstract}
Gaussian quantum Markov semigroups are the natural non-commutative extension of classical Ornstein-Uhlenbeck semigroups.
They arise in open quantum systems of bosons where canonical non-commuting random variables
of positions and momenta come into play.
If there exits a faithful invariant density we explicitly compute the optimal exponential convergence rate,
namely the spectral gap of the generator, in non-commutative $L^2$ spaces determined by the invariant density
showing that the exact value is the lowest eigenvalue of a certain matrix determined by the diffusion and drift matrices.
The spectral gap turns out to depend on the non-commutative $L^2$ space considered, whether the
one determined by the so-called GNS or KMS multiplication by the square root of the invariant density.
In the first case, it is strictly positive if and only if there is the maximum number of linearly independent
noises. While, we exhibit explicit examples in which it is strictly positive only with KMS multiplication.
We do not assume any symmetry or quantum detailed balance condition with respect to the invariant density.
\end{abstract}

\section{Introduction}\label{sec:intr}
Classical Ornstein-Uhlenbeck semigroups $(T_t)_{t\geq 0}$ on bounded measurable functions on
$\mathbb{R}^d$ have the following explicit formula
\[
(T_tf)(x) = \frac{1}{\sqrt{(2\pi)^d\operatorname{det}(\Sigma_t)}}
\int_{\mathbb{R}^d} f\left(\mathrm{e}^{tZ^*}x-y\right)\exp\left(-\langle y, \Sigma_t^{-1} y\rangle/2\right)\mathrm{d}y
\]
where $Z$ is the drift matrix and $\Sigma_t$ the covariance matrix at time $t$ given by
$\Sigma_t = \int_0^t \mathrm{e}^{sZ^*} C \mathrm{e}^{sZ}\mathrm{d}s$ in terms of the symmetric matrix $C$
of diffusion coefficients. Here, for simplicity, we assume that $\Sigma_t$ is invertible and that the mean
of the Markov process is zero.
If we consider the function $x\mapsto \exp(\mi\langle z,x\rangle)$ the above formula becomes
\[
T_t \exp(\mi\langle z,\cdot\rangle) = \exp\left(-\frac{1}{2}\int_0^t\left\langle \mathrm{e}^{sZ}z, C \mathrm{e}^{sZ}z\right\rangle\mathrm{d}s\right) \exp(\mi\langle \mathrm{e}^{tZ}z,\cdot\rangle)\, .
\]
In applications to quantum theory one has to deal with random variables that do not commute, typically positions and momenta,
vector spaces are complex and symplectic structures come into play, but a similar formula holds (see \cite{Demoen77,Vheu78} and
the references therein). More precisely, for a system of $d$ bosons, the natural variables are $d$ position $q_1,\dots, q_d$ and
$d$ momentum $p_1=-\mi\mathrm{d}/\mathrm{d}q_1, \dots, p_d=-\mi\mathrm{d}/\mathrm{d}q_d$ operators. Exponential functions
are replaced by unitary Weyl operators
$W(x+\mi y) = \exp\left(-\mi\sqrt{2}\sum_{k=1}^d(x_kp_k-y_kq_k)\right)$ ($x,y\in\mathbb{R}^d$) and one finds the explicit
formula (\ref{eq:explWeyl}), Section \ref{sect:GQMS} (with the further inclusion of a non-zero mean vector $\zeta\in\mathbb{C}^d$). There, $z=x+\mi y$, $Z,C$ are \emph{real linear} operators
on $\mathbb{C}^d$ that can be viewed as complex linear operators on $\mathbb{C}^{2d}$ as explained in Section \ref{sect:GQMS},
Remark \ref{rem:complexify}.
In this way, one can define a Gaussian quantum Markov semigroup (QMS), namely a weakly$^*$-continuous semigroups of
completely positive, identity preserving maps on the von Neumann algebra of all bounded operators a Fock space
(see Section \ref{sect:GQMS} for precise definitions).

QMSs are essential tools in the mathematical modeling of open quantum systems (see \cite{BrPe,Partha}). Starting from the seminal
papers by Gorini, Kossakowski, Lindblad and Sudharshan (GKLS) in the seventies, they have been used in the physical literature (\cite{Demoen77,Vheu78}) and are now established as non-commutative extensions of classical Markov semigroups
(see \cite{CaMa,FaCiLi,Frigerio-irreducibility,VeWi,WiZhang}.
Gaussian QMSs, that are considered in this article, are a notable class of semigroups for many reasons: they include many  semigroups
that are interesting for physical applications (see \cite{BrPe,FlynnCoViola,Teret} and the references therein),
they are the natural non-commutative extension of classical Ornstein-Uhlenbeck semigroups
(see Section \ref{sect:GQMS} and \cite{AFP,AFPnT,CaSa,CaMa,ChMi,DaPratoZab,FaCiLi,Teret}), they arise from quantum master equations
and Fokker-Planck equations \cite{ArnoldCarlenJu,Arnold-Sig} and they provide a well-behaved and workable class
of semigroups acting on the algebra $\mathcal{B}(\mathsf{h})$ of all bounded operators on a Hilbert space $\mathsf{h}$.
Furthermore, they have countless applications to open quantum systems of bosons,
but rather little is known about their properties, especially if one thinks of the wealth of results established in the
classical case (see \cite{LuPaMe,OtPaPr} and the references therein).

Several properties of gaussian QMS have been recently established: the characterization as the unique class
of quantum Markov semigroups preserving quantum gaussian states \cite{Po2022}, the structure of the subalgebra where they
act as $^*$-homomorphisms \cite{AFP}, the characterization of irreducibility (under a regularity assumption \cite{FaPo-IDAQP2022}) and,
in the symmetric case, gradient flow properties and entropy inequalities \cite{CaMa}, and the expression of the generator as a
sum of squares of derivations \cite{VeWi}.

The main result of this paper is the explicit calculation of the spectral gap of a gaussian QMS as the lowest eigenvalue of
a certain matrix determined by the diffusion and drift parameters (Theorem \ref{thm:spectralGap}).
As for the classical Ornstein-Uhlenbeck process, assuming that the matrix $Z$ is stable, namely all eigenvalues
have strictly negative real part, every initial state converges to the invariant one  \cite{FaPo-IDAQP2024}.
It is therefore natural to study the speed of convergence towards that invariant state
in view of applications to problems such as  return to equilibrium and
limit theorems and dynamical phase transitions (see \cite{GiVHCaGu}).

A natural approach, motivated by the analogue in classical Markov semigroups theory, is to embed the set of bounded operators in the
Hilbert space of Hilbert-Schmidt operators, $\mathfrak{I}_2 (\h)$, by multiplying them by the square root of the invariant density.
In this way, one obtains a contraction semigroup (see \cite{CaFa}) and one can analyse its
generator by spectral analytic tools. However, due to the non-commutative nature of operators, multiplication by the
invariant density can be done in multiple ways, obtaining different embeddings. In this paper, first of all, we work with
GNS embedding
\[
i_2:\B\hookrightarrow\mathfrak{I}_2(\mathsf{h}),\quad\ i_2(x)=x\rho^{1/2}
\] We also consider the KMS embedding (see \cite{CaFa})
\[
i_{1/2,2}:\B\hookrightarrow\mathfrak{I}_2(\mathsf{h}),\quad\  i_{1/2,2}( x) = \rho^{1/4} x \rho^{1/4}
\]
and show that, in this case we get a different result (Theorem \ref{thm:spectralGapKMS}) for the explicit
expression of the spectral gap. In particular, the spectral gap could be strictly positive for
the KMS embedding and zero for the GNS embedding (see the example in subsection \ref{subsect:KMSgap}).
This is a typical feature of non-commutativity. Our results are first presented for the $i_2$
embedding because computations are somewhat more direct from the explicit formula of the characteristic function of
a quantum gaussian state.

Once chosen the GNS embedding, since $i_2$ has dense range, the next step is to extend $\T$ to a strongly continuous
contraction semigroup $(T_t)_{t \geq 0}$ on $\HS$ by setting
\[
T_t(x \rho^{1/2}) = \mathcal{T}_t(x)\rho^{1/2}\qquad\forall\,x\in\B,\ t\geq 0.
\]
Clearly $\rho^{1/2}$ is an invariant vector for $(T_t)_{t\geq 0}$. Denoting by $\langle x,y\rangle_{2}=\tr(x^*y)$ the
Hilbert-Schmidt scalar product and by $\Vert\cdot\Vert_2=\langle x,x\rangle_{2}^{1/2}$ the norm, we define the \emph{spectral gap} of this semigroup in the following way
\begin{definition}
If there exists $g>0$ such that
\begin{equation}\label{def-g}
\left\Vert T_t(x)-\langle\rho^{1/2},x\rangle_2\,\rho^{1/2}\right\Vert_2 \leq
\mathrm{e}^{-g t}\left\Vert x-\langle\rho^{1/2},x\rangle_2\,\rho^{1/2}\right\Vert_2 \quad\forall\, x\in \mathfrak{I}_2(\mathsf{h}),\,\forall\, t >0
\end{equation}
we say the semigroup $T$ has a spectral gap.
In that case, we call the \emph{spectral gap} of the semigroup the biggest constant $g>0$ that satisfies \eqref{def-g}.
\end{definition}

The spectral gap for the KMS embedding is defined in a similar way.

Note that, as usual in the theory of Markov processes and with a little abuse of terminology, the name
spectral gap for the above constant $g$, should be more precisely referred to as the spectrum of the symmetric
part of the generator of $T$. In this case, indeed, it is the distance from $0$ of the remaining part of the spectrum.

The spectral gap $g$ is explicitly computed for both the above embeddings (Theorem \ref{thm:spectralGap}, and Theorem
\ref{thm:spectralGapKMS}). In the first case, it turns out to be strictly positive under the
assumption that there is noise in all the non-commutative coordinates. In the second case, this condition in no
more necessary. In particular, we give explicit examples in which it strictly positive only for the
KMS embedding in subsection \ref{subsect:KMSgap}.

It is important to note that, due to the nature of gaussian QMSs, the computation of the spectral gap,
as other problems (see \cite{AFP,AFPnT,AFPKoss}), can be reduced to linear algebraic one.
However, it is worth noticing that the final result is not a straightforward generalization of
the classical one because one can find classical Ornstein-Uhlenbeck subprocesses with strictly positive spectral
gap of a quantum Ornstein-Uhlenbeck process with $g=0$ (see the example at the end of Section \ref{sec:oneDimCase}).
\smallskip

The paper is organized as follows. Section \ref{sect:GQMS} is devoted to the formal introduction of all the necessary
mathematical preliminaries, including the definition of gaussian QMSs. In Section \ref{sec:OUQMS} we
discuss the relationship between gaussian QMSs and classical Ornstein-Uhlenbeck semigroups, showing how the latter emerge
from suitable restrictions of the former. Section \ref{sec:spectralGap} contains the main result
on the spectral gap for the GNS embedding.
 Section \ref{sec:invariantStates} contains some further insight on the sufficient conditions used for the spectral gap, showing that they are not very restrictive when looking for a positive spectral gap.
In Section \ref{sec:oneDimCase} we present a class of examples with $d=1$ depending on some parameters.
The study of the spectral gap for the KMS embedding is undertaken in Section \ref{sect:KMSembed}.
More technical material and some computations are collected in Appendixes A, B and C.

\section{Gaussian Quantum Markov Semigroups} \label{sect:GQMS}
In this section we introduce gaussian QMSs starting from their generators and fix some notation.

Let $\mathsf{h}$ be the Fock space $\mathsf{h}=\Gamma(\mathbb{C}^d)$ which is
isometrically isomorphic to $\Gamma(\mathbb{C})\otimes\cdots\otimes\Gamma(\mathbb{C})$ with canonical
orthonormal basis $(e(n_1,\ldots,n_d))_{n_1,\ldots,n_d\geq0}$
(with $e(n_1,\ldots,n_d)$ the symmetrized version of $e_{n_1}\otimes\ldots\otimes e_{n_d}$). Let $a_j, a_j^{\dagger}$ be the creation and
annihilation operators of the Fock representation of the $d$-dimensional Canonical Commutation Relations (CCR)
\begin{eqnarray*}
a_j\,e(n_1,\ldots,n_d)& =& \sqrt{n_j}\ e(n_1,\ldots,n_{j-1},n_j-1,\ldots,n_d),\\
a_j^{\dagger}\,e(n_1,\ldots,n_d)&=&\sqrt{n_j+1}\ e(n_1,\ldots,n_{j-1},n_j+1,\ldots,n_d),
\end{eqnarray*}
The CCRs are written as $[a_j,a_k^{\dagger}]=\delta_{jk}\mathbbm{1}$, where $[\cdot,\cdot]$ denotes the
commutator, or, more precisely, $[a_j,a_k^{\dagger}]\subseteq \delta_{jk}\mathbbm{1}$ because the
domain of the operator in the left-hand side is smaller.

%For any $g\in\mathbb{C}^d$, define the exponential vector $e(g)$ associated with $g$ by
% \[
% e_g=\sum_{n\in \mathbb{N}^d}\frac{g_1^{n_1}\cdots g_d^{n_d}}{\sqrt{n_1!\cdots n_d!}}\,e(n_1,\ldots,n_d)
% \]
%Creation and annihilation operators with test vector $v\in\mathbb{C}^d$ can also be defined on the
%total set of exponential vectors (see \cite{Pa}) by
% \[
% a(v)e_g=\langle v,g \rangle e_g, \quad a^{\dagger}(v)e_g=\frac{\mathrm{d}}{\mathrm{d}\varepsilon}e_{g+\varepsilon u}|_ {\varepsilon=0}
% \]
%for all $u\in\mathbb{C}^d$. The unitary correspondence $\Gamma(\mathbb{C}^d)\mapsto\Gamma(\mathbb{C})\otimes\cdots
%\otimes\Gamma(\mathbb{C})$
%\[
%e_g\mapsto\sum\limits_{n_1\geq0,\ldots, n_d\geq0}\frac {g_1^{n_1}\ldots g_d^{n_d}}{\sqrt{n_1!\ldots n_d!}}
%\,e_{n_1}\otimes\ldots\otimes e_{n_d}
%\]
Linear combinations of both creation and annihilation operators are denoted as follows:
 \[
 a(v)=\sum\limits_{j=1}^d\overline{v}_ja_j, \quad a^{\dagger}(u)=\sum\limits_{j=1}^d u_ja_j^{\dagger}
 \]
for all $u^{T}=[u_1,\ldots,u_d],v^{T}=[v_1,\ldots,v_d]\in\mathbb{C}^d$.

The above operators are obviously defined on the linear manifold $D$ spanned by the
elements $(e(n_1,\ldots,n_d))_{n_1,\ldots,n_d\geq0}$ of the canonical orthonormal basis of $\mathsf{h}$ that
turns out to be an essential domain for all the operators considered so far. This also happens for field operators
\begin{equation}\label{eq:quadrature}
q(u) = \left(a(u)+a^\dagger(u)\right)/\sqrt{2}\qquad u\in\mathbb{C}^d
\end{equation}
that are symmetric and essentially self-adjoint on the domain $D$ by Nelson's theorem
on analytic vectors (\cite{ReSi} Th. X.39 p. 202).  If the vector $u$ has real (resp. purely
imaginary) components one finds position (resp. momentum) operators and the commutation relation
$[q(u),q(v)]\subseteq \mi \Im \langle u, v\rangle \unit$ (where $\Im$ and $\Re$ denote the imaginary and real part
of complex number and vectors). Momentum operators, i.e. fields $q(\mi r)$ with $r\in\mathbb{R}^d$ are also denoted
by $p(r)=\sum_{1\leq j\leq d} r_j p_j$ where $p_j=\mi(a^\dagger_j-a_j)/\sqrt{2}$. In a similar way we write
$q(r)=\sum_{1\leq j\leq d} r_j q_j$ with $q_j=q(e_j)=(a^\dagger_j+a_j)/\sqrt{2}$.

Another set of operators that will play
an important role in this paper are the Weyl operators $W(z), z \in \mathbb{C}^d$, which generate the entirety of $\mathcal{B}(\mathsf{h})$. They satisfy the CCR in the
exponential form, namely, for every $z,z^\prime \in \mathbb{C}^d$,
\begin{equation} \label{eq:WeylCCR}
	W(z)W(z^\prime) = \me^{-\mi \Im \scalar{z}{z^\prime}} W(z+z^\prime).
\end{equation}
It is well-known that $W(z)$ is the exponential of the anti self-adjoint operator $-\mi\sqrt{2}\, q(\mi z)$
\begin{equation} \label{eq:weylQ}
W(z) = \mathrm{e}^{-\mi\sqrt{2}\, q(\mi z)} = \mathrm{e}^{z a^\dagger-\overline{z}a}.
\end{equation}

A QMS $\mathcal{T}=(\mathcal{T}_t)_{t\geq 0}$ is a weakly$^*$-continuous semigroup of completely positive,
identity preserving, weakly$^*$-continuous maps on $\mathcal{B}(\mathsf{h})$.
The predual semigroup $\mathcal{T}_*= (\mathcal{T}_{*t})_{t\geq 0}$ on the predual space of trace class
operators on $\mathsf{h}$ is a strongly continuous contraction semigroup.

Gaussian QMSs can be defined either through their explicit action on Weyl operators or through their
generator, as we will do.
Let $L_\ell, H$ be the operators on $\mathsf{h}$ defined on the domain $D$ by
	\begin{align}
		H&= \sum_{j,k=1}^d \left( \Omega_{jk} a_j^\dagger a_k + \frac{\kappa_{jk}}{2} a_j^\dagger a_k^\dagger + \frac{\overline{\kappa_{jk}}}{2} a_ja_k \right) + \sum_{j=1}^d \left( \frac{\zeta_j}{2}a_j^\dagger + \frac{\conj{\zeta_j}}{2} a_j \right), \label{eq:H}\\
		 L_\ell &= \sum_{k=1}^d \left( \overline{v_{\ell k}} a_k + u_{\ell k}a_k^\dagger\right) \label{eq:Lell}
	\end{align}
where $1 \leq m \leq 2d$, $\Omega:=(\Omega_{jk})_{1\leq j,k\leq d} = \Omega^*$ and $\kappa:= (\kappa_{jk})_{1\leq j,k\leq d}= \kappa^{T}
\in M_d(\mathbb{C})$, are $d\times d$ complex matrices with $\Omega$ Hermitian and $\kappa$ symmetric,
$V=(v_{\ell k})_{1\leq \ell\leq m, 1\leq  k\leq d}, U=(u_{\ell k})_{1\leq \ell\leq m, 1\leq  k\leq d} \in M_{m\times d}(\mathbb{C})$
are $m\times d$ matrices and $\zeta=(\zeta_j)_{1\leq j\leq d} \in \mathbb{C}^d$. We assume also $\ker (V^*) \cap \ker (U^T) = \{0\}$ so that
the operators $L_\ell$, called Kraus' operators,  are linearly independent (see \cite{AFP}, Proposition 2.2).The above operators are closable and we will denote their closure by the same symbol.

For all $x\in\mathcal{B}(\mathsf{h})$ consider the quadratic form with domain $D\times D$
\begin{equation}\label{eq:Lform}
\begin{split}
\texttt{\textrm{\pounds}}(x)  [\xi',\xi] &=  \mi\scalar{H\xi'}{x\xi} - \mi\scalar{\xi'}{x H\xi} \\
			&-  \frac{1}{2} \sum_{\ell=1}^m \left( \scalar{\xi'}{xL_\ell^*L_\ell \xi} -2\scalar{L_\ell \xi'}{xL_\ell \xi}
         + \scalar{L_\ell^* L_\ell \xi'}{x \xi} \right)
\end{split}
\end{equation}This is a natural way to make sense of a Gorini, Kossakowski, Lindblad-Sudarshan (GKLS) representation
of the generator (see \cite{Po2022} Theorems 5.1, 5.2 and also  \cite{Demoen77,Vheu78}) in a generalized form since operators
$L_\ell, H$ are unbounded. Indeed, in other models, when these operators are bounded, one writes
the generator as
\begin{equation} \label{eq:Lgen}
\mathcal{L}(x) = \mi\left[ H, x\right]
-\frac{1}{2}\sum_{\ell=1}^m \left( L_\ell ^*L_\ell\, x - 2 L_\ell^* x L_\ell + x\, L_\ell ^*L_\ell\right),
\qquad x\in{\mathcal{B}(\mathsf{h})}.
\end{equation}

Gaussian QMSs are then defined by the following Theorem, whose proof can be found in \cite{AFPnT}, Appendix A.

\begin{theorem}\label{th:G-QMS!}
There exists a unique QMS, $\mathcal{T}=(\mathcal{T}_t)_{t\geq 0}$ such that, for all $x\in\mathcal{B}(\mathsf{h})$ and
$\xi,\xi'\in D$, the function $t\mapsto  \scalar{\xi'}{\mathcal{T}_t (x) \xi}  $ is differentiable and
\begin{align*}
	\frac{\mathrm{d}}{\mathrm{d}t} \scalar{\xi'}{\mathcal{T}_t (x) \xi}  = \texttt{\textrm{\pounds}}(\mathcal{T}_t(x)) [\xi',\xi]
 \qquad \forall\, t\geq 0.
\end{align*}
\end{theorem}

One can derive as in \cite{AFPnT}, Theorem 2.4, the following formula for the action of a gaussian semigroup on Weyl operators.

\begin{theorem}\label{th:explWeyl}
Let $(\mathcal{T}_{t})_{t\ge 0}$ be the quantum Markov semigroup with generalized GKLS  generator
associated with $H,L_\ell$ as above. For all Weyl operator $W(z)$ we have
\begin{equation}\label{eq:explWeyl}
\mathcal{T}_t(W(z))
= \exp\left(-\frac{1}{2}\int_0^t \Re\scalar{\mathrm{e}^{sZ}z}{
  C \mathrm{e}^{sZ}z}\mathrm{d}s
+\mi\int_0^t  \Re\scalar{\zeta}{\mathrm{e}^{sZ}z} \mathrm{d}s \right)
W\left(\mathrm{e}^{tZ}z\right)
\end{equation}
where the \emph{real linear} operators $Z,C$ on $\mathbb{C}^d$ are
\begin{align} \label{eq:Zdef}
	Zz &= \left[ \left(U^T\, \conj{U} - V^T\, \conj{V} \right)/2 + \mi \Omega \right]z
+ \left[ \left(U^T V - V^T U\right)/2 + \mi \kappa \right] \conj{z}, \\
	Cz &= \left(U^T\, \conj{U} + V^T\, \conj{V} \right)z + \left(U^T V + V^T U\right)\conj{z}. \label{eq:Cdef}
\end{align}
\end{theorem}

Note that the operators $Z, C$ that appear in the previous Theorem are only real linear since they inherit
the real vector space structure of $\mathbb{C}^d$ as argument of the Weyl operators.

For $d=1$ special cases of gaussian QMSs have been considered as quantum analogues of classical Ornstein-Uhlenbeck semigroups (see \cite{CaSa, FaCiLi}). The multidimensional case has also been extensively studied \cite{CaMa,VeWi,WiZhang} under symmetry or detailed balance conditions.

The natural duality between  $\mathcal{B}(\mathsf{h})$ and trace class operators
on $\mathsf{h}$ given by $(\rho,x)\mapsto \tr (\rho\, x)$, together with the weak$^*$ continuity
of the QMS $(\mathcal{T}_t)_{t\geq 0}$ allows one to define the pre-adjoint QMS $(\mathcal{T}_{*t})_{t\geq 0}$
on trace class operators on $\mathsf{h}$. Operators $\mathcal{T}_{*t}$ act on positive operators with
unit trace, called density matrices and representing quantum states, that play the role of classical probability densities
on $\mathbb{R}^{2d}$. A state is called \emph{faithful} if $\rho\,\xi=0$ implies $\xi=0$; this is the analogue of
a classical density with full support.

One of the defining properties of gaussian QMSs is that operators $\mathcal{T}_{*t}$ preserve the set of quantum gaussian states.
A quantum state is gaussian if its quantum characteristic function $\widehat{\rho}\,$
has a similar expression to the one of classic gaussian multivariate random variables
\begin{equation}\label{funz-caratt-rho}
\widehat{\rho}(z):=\tr(\rho W(z))=\exp\left(-\mi\Re\scalar{\mu}{z}-\frac{1}{2}\Re\scalar{z}{Sz} \right)\qquad\forall\,z\in\mathbb{C}^d
\end{equation}
for some $\mu\in\h$ and real linear, bounded, positive and invertible operator $S$ on $\h$ (we write $S\in\mathcal{B}_{\mathbb{R}}(\h)$). We denote this state as $\rho=\rho_{{(\mu, S)}}$. One can prove as in \cite{Po2022} Theorem 5.1 that

\begin{proposition} \label{prop:gaussianStateEvolution}
	If $\mathcal{T}$ is a gaussian QMS, then $\mathcal{T}_{*t} (\rho_{(\mu, S)}) = \rho_{(\mu_t, S_t)}$ with
	\begin{equation} \label{eq:invariantParameterEvolution}
		\mu_t = \me^{tZ^\sharp}\mu - \int_0^t \me^{sZ^\sharp} \zeta \md s, \quad S_t = \me^{tZ^\sharp} S \me^{tZ} + \int_0^t \me^{sZ^\sharp} C \me^{sZ} \md s
	\end{equation}
where $Z^\sharp$ denotes the adjoint of   $Z$ with respect to the scalar product $\Re\langle\cdot,\cdot\rangle$.
\end{proposition}
\medskip

The adjoint $Z^\sharp$ of $Z$ with respect to the scalar product $\Re\langle\cdot,\cdot\rangle$
is explicitly given by $Z^\sharp z = Z_1^* z + Z_2^T \conj{z}$ where $Z_1$ and $Z_2$ are the operators in (\ref{eq:Zdef})
acting on $z$ and $\overline{z}$. Here, however, further clarifications and remarks on properties of real linear operators
that are useful to better explain the computations performed in the following sections are in order.

\begin{remark}\label{rem:complexify}\rm In general a real linear operator $A$ on $\mathbb{C}^d$ can be written as
\[
	Az = A_1 z + A_2 \conj{z}, \quad \forall z \in \mathbb{C}^d,
\]
where now $A_1, A_2$ are complex linear operators. In order to better understand the action of such operators we can exploit the isomorphism of $(\mathbb{C}^d, \Re\scalar{\cdot}{\cdot})$ as a real vector space with $(\mathbb{R}^{2d}, \scalar{\cdot}{\cdot})$, through the decomposition
of a complex vector $z = x+ \mi y$ in its real and imaginary part. On this new vector space $A$ takes the form
\[
	A_{2d} = \begin{bmatrix}
		\Re A_1 + \Re A_2 & \Im A_2 - \Im A_1 \\
		\Im A_1 + \Im A_2 & \Re A_1 - \Re A_2
	\end{bmatrix}
\]
and $A_{2d}$ is a \emph{real linear} operator on $\mathbb{R}^{2d}$.
This isomorphism will also be useful when considering the complexification of the operator $A$. Indeed it is much simpler to work with the complexification $\mathbf{A}$ of $A_{2d}$
\[
	\mathbf{A} = \begin{bmatrix}
		\Re A_1 + \Re A_2 & \Im A_2 - \Im A_1 \\
		\Im A_1 + \Im A_2 & \Re A_1 - \Re A_2
	\end{bmatrix},
\]
now acting on $(\mathbb{C}^{2d}, \scalar{\cdot}{\cdot})$, than it is to consider the complexification of $A$. For this reason whenever the complexification of a real linear operator is needed we will think of it as acting on $(\mathbb{C}^{2d}, \scalar{\cdot}{\cdot})$. We also have
\begin{equation}\label{scal-isom-2d}
\Re\scalar{z}{Aw}=\scalar{ \begin{bmatrix} \Re z \\ \Im z \end{bmatrix}}{A_{2d}\begin{bmatrix} \Re w \\ \Im w\end{bmatrix}}=\scalar{ \begin{bmatrix} \Re z \\ \Im z \end{bmatrix}}{\mathbf{A}\begin{bmatrix} \Re w \\ \Im w\end{bmatrix}}\quad\forall\,z,w\in\mathbb{C}^d,
\end{equation}
where the scalar products are the canonical ones on $\mathbb{C}^d, \mathbb{R}^{2d}, \mathbb{C}^{2d}$ respectively. The adjoint operations are denoted $A^\sharp, A_{2d}^T, \mathbf{A}^*$, respectively.

A notable example of real linear operator is $J z= - \mi z$ to which is associated
\[
	\mathbf{J} = \begin{bmatrix}
		0 & \unit\, \\ -\unit & 0
	\end{bmatrix}\in M_{2d}(\mathbb{C}).
\]

It is easy to prove that the following equality holds
\begin{equation}\label{legame-prod-scal}
\Re\scalar{z}{Aw}+\mi\Im\scalar{z}{w}=\scalar{ \begin{bmatrix} \Re z \\ \Im z \end{bmatrix}}{(\mathbf{A}+\mi \mathbf{J})\begin{bmatrix} \Re w \\ \Im w\end{bmatrix}}\qquad\forall\,z,w\in\mathbb{C}^d,
\end{equation}
and
\begin{equation}\label{complessificazione}
\scalar{\begin{bmatrix} z_1 \\ z_2 \end{bmatrix}}{\mathbf{A}\begin{bmatrix} w_1 \\ w_2 \end{bmatrix}}=
\Re\scalar{\zeta_1}{A\xi_1}+\Re\scalar{\zeta_2}{A\xi_2}+\mi\Re\scalar{\zeta_1}{A\xi_2}-\mi\Re\scalar{\zeta_2}{A\xi_1}
\end{equation}
for all $z_1,z_2,w_1,w_2\in\mathbb{C}^d$, where
$$\begin{array}{ll}\zeta_1=\Re z_1+\mi \Re z_2 &\
\zeta_2=\Im z_1+\mi\Im z_2\\
\xi_1=\Re w_1+\mi \Re w_2 &\
\xi_2=\Im w_1+\mi\Im w_2
\end{array}.
$$
We can then conclude that two \emph{real} linear operators $A$ and $B$ on $\mathbb{C}^d$ coincide if and only if their complexification $\mathbf{A}$ and $\mathbf{B}$ coincide (as complex linear operators on $\mathbb{C}^{2d}$).

\end{remark}
The following Proposition now collects useful formulae on the operators $Z,C$ and introduces the matrix
\[
	\CZ := \mathbf{C} - \mi(\mathbf{Z}^*\mathbf{J} + \mathbf{J} \mathbf{Z})
\]
which will have a central role in the paper.
\begin{proposition} \label{prop:CandZ}
	Let $Z, C$ be given by \eqref{eq:Zdef} and \eqref{eq:Cdef}.
	\begin{enumerate}
		\item We can write
		\[
		C = \sqrt{C}^T \sqrt{C} \geq 0,
	\]
	where $\sqrt{C}z = \conj{U} z + V \conj{z}$, for $z \in \mathbb{C}^d$.
	\item It holds
	\begin{eqnarray}
        \mathbf{Z} & = & \frac{1}{2}\begin{bmatrix}
			\Re \left(\left( U - \conj{V}\right)^*\left(U +\conj{V}\right)\right) & \Im\left(\left( U - \conj{V}\right)^*\left(U -\conj{V}\right)\right) \\
			- \Im \left(\left( U + \conj{V}\right)^*\left(U +\conj{V}\right)\right) & \Re \left(\left( U + \conj{V}\right)^*\left(U -\conj{V}\right)\right)
		\end{bmatrix} \nonumber \\
 & + & \begin{bmatrix}
			-\Im \left(\Omega + \kappa\right) & \Re\left(\kappa - \Omega \right) \\
			\Re\left(\Omega + \kappa\right) & \Im\left(\kappa - \Omega\right)
		\end{bmatrix} \label{eq:mathbfZ} \\
		\CZ  & = & \begin{bmatrix}
		\left( U + \conj{V} \right)^* \left(U + \conj{V}\right) & -\mi\left( U + \conj{V}\right)^* \left( U - \conj{V} \right) \\
		\mi \left( U - \conj{V}\right)^* \left(U +  \conj{V}\right) & \left( U -\conj{V} \right)^* \left( U - \conj{V}\right)
\end{bmatrix} \geq 0. \label{eq:mathbfCZ}
	\end{eqnarray}
	\item The semigroup has exactly $2d$ (linearly independent) Kraus' operators if and only if $\CZ > 0$.
%	\item {$\mathbf{C} - \mi(\mathbf{Z}^*\mathbf{J} + \mathbf{J} \mathbf{Z}) \geq 0$. Moreover $\mathbf{C} - \mi(\mathbf{Z}^*\mathbf{J} + \mathbf{J} \mathbf{Z})>0$ if and only if $\mathbf{C} + \mi(\mathbf{Z}^*\mathbf{J} + \mathbf{J} \mathbf{Z}) >0$.}
	\end{enumerate}
\end{proposition}
\begin{proof}
Equalities in statements $1$ and $2$ follow from direct computation starting from the expression \eqref{eq:Zdef}, \eqref{eq:Cdef}
(see Appendix A). The inequality instead follows from noticing that, setting
	\[
		M = \begin{bmatrix}
			U + \conj{V} & -\mi \left( U - \conj{V}\right)
		\end{bmatrix} \in M_{m, 2d}(\mathbb{C})
	\]
	one has
	\[
		\CZ = M^* M \geq 0.
	\]
	Eventually the inequality holds strict if and only if $\ker M = \{0\}$ or equivalently $\operatorname{dim}\ran M = 2d$. However we have
	\[
		\operatorname{dim}\ran M = m - \operatorname{dim}(\ran M)^\perp = m-\operatorname{dim} \ker M^* = m,
	\]
	since  $\ker M^* = \ker ( U + \conj{V})^* \cap \ker (U - \conj{V})^*$ and by taking linear combinations of the two operators we get to  $\ker M^* = \ker U^T \cap \ker V^* = \{0\}$ by the conditions we imposed on the matrices $U, V$.
%	
%	{\red Finally statement $3$ is a trivial consequence of $2$ since $\mathbf{C} - \mi(\mathbf{Z}^*\mathbf{J} + \mathbf{J} \mathbf{Z}) = \overline{\mathbf{C} + \mi(\mathbf{Z}^*\mathbf{J} + \mathbf{J} \mathbf{Z})}$ being $\mathbf{C}, \mathbf{Z} $ and $\mathbf{J}$ real matrices.}
%	
\end{proof}

Note that we introduced the notation $\CZ$ using bold-face characters even though the operator is not the complexification of any real linear operator. This is to emphasize the fact that the operator acts on $\mathbb{C}^{2d}$.

\medskip

We end this section by recalling the following result on invariant densities, namely, those positive
operators $\rho$ with unit trace such that $\mathcal{T}_{*t}(\rho)=\rho$ for all $t\geq 0$. This parallels known
results (\cite{DaPratoZab} Section 11.2.3) for classical Ornstein-Uhlenbeck processes. The real linear operator
$\mathbf{Z}$ is called \emph{stable} if all its eigenvalues have strictly negative real part. It is clear from
formula \eqref{scal-isom-2d} that $Z$ is stable if and only if $\mathbf{Z}$ is stable.

\begin{theorem} \label{thm:invariantState}
	Let $\mathcal{T}$ be a gaussian QMS and suppose $Z$ is stable.
  Then there is a unique gaussian state $\rho = \rho_{(\mu,S)}$ which is invariant for $\mathcal{T}$ whose parameters are given by:
	\begin{equation} \label{cond-su-S}
		\mu = (Z^\sharp)^{-1} \zeta, \quad S = \int_0^\infty \me^{sZ^\sharp} C \me^{sZ} \md s.
	\end{equation}
	If also $\CZ >0$ then $\rho$ is faithful and it is the unique invariant state for $\mathcal{T}$.
\end{theorem}
\begin{proof}
	Let us start by showing the existence of an invariant state under the stability condition.
	Let $\rho= \rho_{(\mu, S)}$ a gaussian state. For it to be an invariant state, from Proposition \ref{prop:gaussianStateEvolution},
we must have
	\[
		\mu = \me^{tZ^\sharp}\mu - \int_0^t \me^{sZ^\sharp} \zeta \md s, \quad S_t = \me^{tZ^\sharp} S \me^{tZ} + \int_0^t \me^{sZ^\sharp} C \me^{sZ} \md s, \quad \forall t \geq 0.
	\]
	The previous equations can be equivalently rewritten as
	\[
		\int_0^t \me^{sZ^\sharp} \left( Z^\sharp \mu - \zeta\right) \md s = 0, \quad \int_0^t \me^{sZ^\sharp} \left( Z^\sharp S + SZ + C\right) \me^{sZ} \md s = 0, \quad \forall t \geq 0.
	\]
	Since the equations hold for every $t \geq 0$ their derivative must vanish at every $t \geq 0$. Moreover both $\me^{tZ^\sharp}$ and $\me^{tZ}$ are invertible so the equations are eventually equivalent to
	\begin{equation} \label{eq:invStateEquation}
		Z^\sharp \mu = \zeta, \quad Z^\sharp S + S Z = -C.
	\end{equation}
	Now it is easy to show that the pair $\mu, S$ given by \eqref{cond-su-S} solves the previous equations. Indeed the integral in the definition for $S$ converges and
	\begin{align*}
		Z^\sharp S + S Z = \int_0^\infty \me^{sZ^\sharp} \left( Z^\sharp C + CZ\right) \me^{sZ} \md s = \left[\me^{sZ^\sharp} C \me^{sZ}\right]_0^\infty = -C.
	\end{align*}
	where the last equality is due to the stability of $Z$.
	
	Uniqueness of the invariant state among gaussian states is due to the fact that $Z$ is stable and rewriting equation \eqref{eq:invStateEquation} as a linear system for the entries of $S$ (see \cite{Bellman} Theorem 12.4 and the discussion preceding it) one obtains a system matrix which is invertible.
	
	Suppose now that $\CZ > 0$. It is easy to see that $\CZ$, explicitly given in Proposition \ref{prop:CandZ},
is unitarily equivalent to the Kossakowski matrix of the semigroup (see \cite{AFPKoss}). Therefore, by \cite{AFPKoss} Theorem 4, the semigroup is irreducible. \\
	Recall now that a gaussian state is faithful if and only if $\mathbf{S} - \mi \mathbf{J} >0$, or equivalently, by conjugation,
 $\mathbf{S} + \mi \mathbf{J}>0$ (see \cite[Theorem 4]{KRPSymm}) and note that, thanks to the stability of $Z$ we can write
	\[
		-J = \int_0^{\infty}  \me^{sZ^\sharp} (Z^\sharp J + JZ) \me^{sZ} \md s.
	\]  Using \eqref{cond-su-S} we have therefore
	\begin{align*}
	\mathbf{S} + \mi \mathbf{J} &= \int_0^\infty \me^{s\mathbf{Z}^*} \mathbf{C} \me^{s\mathbf{Z}} \md s  - \mi \int_0^\infty \me^{s\mathbf{Z}^*} (\mathbf{Z}^* \mathbf{J} + \mathbf{JZ}) \me^{s\mathbf{Z}} \md s \\
	&= \int_0^\infty \me^{s\mathbf{Z}^*} \CZ \me^{s\mathbf{Z}} \md s.
	\end{align*}
	Since, from Proposition \ref{prop:CandZ}, $\CZ \geq 0$, we have that $\scalar{w}{\left(\mathbf{S} + \mi \mathbf{J}\right) w} = 0$ for some $w \in \mathbb{C}^{2d}$ if and only if
	\[
		\scalar{\me^{s\mathbf{Z}}w}{ \CZ \me^{s\mathbf{Z}}w} = 0, \quad \forall s \geq 0.
	\]
	which in turn is equivalent to $\me^{tZ}w \in \ker \CZ$ for every $t \geq 0$. However $\CZ > 0$ means this cannot happen.
	We can now end the proof by observing that irreducibility and the existence of an invariant faithful state imply the invariant state is unique (see \cite{Frigerio-irreducibility}).
\end{proof}

\section{Relationship with Classical Ornstein-Uhlenbeck semigroups} \label{sec:OUQMS}
Special cases of gaussian QMSs have been considered as quantum analogues of classical Ornstein-Uhlenbeck
semigroups (see \cite{CaSa, FaCiLi}) for $d=1$. The multidimensional case, under the further assumption of a detailed
balance condition, has been extensively studied \cite{CaMa,VeWi}.
The key observation is that it is possible to restrict a QMS on $\mathcal{B}(\Gamma(\mathbb{C}^d))$ described in Theorem \ref{th:explWeyl}
to an abelian subalgebra generated by commuting (i.e. classical) observables.
Here we show that, with this procedure, one can recover some classical Ornstein-Uhlenbeck semigroups starting from gaussian QMSs.
A preliminary clarification is now in order. Using Schr\"odinger representation (see Example 5.2.16 in \cite{BrRo})
we can represent $\mathcal{B}(\mathsf{h})$ onto $L^2(\mathbb{R}^d)$ via the identifications
\begin{gather*}
	W(cf_j) \psi(x_1, \dots, x_d) = \me^{\mi c f_j} \psi(x_1, \dots, x_d), \\
	W(\mi c f_j) \psi(x_1, \dots, x_d) = \psi(x_1, \dots,x_j - c, \dots x_d) ,
\end{gather*}
with $c \in \mathbb{R}$ and $(f_j)$ the canonical orthonormal basis of $\mathbb{R}^d$.
Under this representation we have then
\begin{gather*}
		q_j \psi (x_1, \dots, x_d) = x_j \psi(x_1, \dots, x_d), \\
		p_j \psi (x_1, \dots, x_d) = -\mi \partial_j \psi(x_1, \dots, x_d).
\end{gather*}
When position operators, in the different coordinates, are the commuting observables of interest, the algebra they generate
is particularly simple to describe. Indeed a (bounded) function of field operators is simply the multiplication operator
for that function, i.e.
\[
	f(q_1, \dots, q_d) \psi (x_1, \dots, x_d) = f(x_1, \dots, x_d) \psi(x_1, \dots, x_d).
\]

We restrict ourselves to the case where $H=0$ and $U,V \in M_{m\times d}(\mathbb{R})$, i.e. all the
coefficients of the Kraus operators are real. This is not at all a necessary restraint but it is sufficient to show
the link between classical Ornstein-Uhlenbeck and gaussian quantum Markov semigroups.
Indeed we can rewrite the Kraus' operators \eqref{eq:Lell} as
\[
	L_\ell = \sum_{j=1}^d \left(r_{\ell j} q_j + \mi s_{\ell j} p_j\right), \quad r_{\ell j} = \frac{\conj{v_{\ell j}} + u_{\ell j}}{\sqrt{2}}, s_{\ell j} = \frac{\conj{v_{\ell j}}-u_{\ell j}}{\sqrt{2}}
\]
with $X=(s_{jk})_{jk}, R=(r_{jk})_{jk} \in M_{md}(\mathbb{R})$, and obtain the following.
\begin{proposition} \label{prop:QMSrestriction}
	Let $H=0$ and $U,V$ be real matrices, then the commutative algebra of bounded functions of $q_1, \dots, q_d$ is invariant for the semigroup and the generator acts on them as
	\[
		\texttt{\textrm{\pounds}}f(q_1, \dots, q_d) =  \frac{1}{2} \sum_{j,k=1}^d(X^*X)_{jk} (\partial_j \partial_k f)(q_1, \dots, q_d) -\sum_{j,k=1}^d(R^*X)_{jk}q_j (\partial_k f) (q_1, \dots, q_d),
	\]
	where $X^*X$ and $R^*X$ correspond, respectively, to the bottom-right $d \times d$ blocks of $\CZ$ and $\mathbf{Z}$, \eqref{eq:mathbfZ} \eqref{eq:mathbfCZ} .
\end{proposition}
\begin{proof}
	All the calculations in this proof are assumed to be performed using the quadratic form on $D \times D$ but, in order to avoid cluttering notation, we leave it implicit.
	On every bounded function of $q_1, \dots, q_d$, again at least on the dense domain $D$,  we can compute
	\[
		[q_j, f(q_1, \dots, q_d)] = 0, \quad [p_j, f(q_1, \dots, q_d)]= -\mi (\partial_j f) (q_1, \dots, q_d).
	\]
	Using these computations we can evaluate the generator on such functions starting from \eqref{eq:Lgen}
	\begin{align*}
		\texttt{\textrm{\pounds}}\left(f(q_1, \dots, q_d)\right) &= -\frac{1}{2} \sum_{\ell=1}^m \left(L_\ell^*[L_\ell, f(q_1, \dots, q_d)] - [L_\ell^*, f(q_1, \dots, q_d)] L_\ell  \right)  \\
		&= -\frac{1}{2}\sum_{\ell = 1}^m \sum_{j,k=1}^d \left(r_{\ell j} q_j - \mi s_{\ell j}p_j\right)\left[r_{\ell k} q_k + \mi s_{\ell k} p_k, f(q_1, \dots, q_d)\right] \\
	&+ \frac{1}{2}\sum_{\ell = 1}^m \sum_{j,k=1}^d\left[ r_{\ell j} q_j - \mi s_{\ell j}p_j, f(q_1, \dots, q_d)\right] \left( r_{\ell k} q_k + \mi s_{\ell k}p_k\right) \\
	&= -\sum_{j,k=1}^d(R^*X)_{jk} q_j(\partial_k f) (q_1, \dots, q_d) \\
	&+ \frac{1}{2} \sum_{j,k=1}^d(X^*X)_{jk} (\partial_j \partial_k f)(q_1, \dots, q_d).
	\end{align*}
\end{proof}

The previous result shows indeed that starting from a particular subset of gaussian QMSs one can recover a Ornstein-Uhlenbeck semigroup.
What one can also show, in the special case of a non-degenerate diffusion matrix $X^*X$, that a converse result also holds.
\begin{proposition}
	Let
	\[
		\mathscr{L} = \frac{1}{2} \sum_{j,k=1}^dQ_{jk} \partial_j \partial_k + \sum_{j,k=1}^d A_{jk} \partial_k,
	\]
	be a Ornstein-Uhlenbeck generator with $Q>0$. Then the semigroup generated by $\mathscr{L}$ can be recovered as a restriction of a suitable gaussian QMS.
\end{proposition}
\begin{proof}
	From Proposition \ref{prop:QMSrestriction} we just need to show that there is a suitable choice of $R, X$ (corresponding uniquely to a choice of $U, V$) such that
	\[
		Q = X^*X, \quad A=-R^*X.
	\]
	However, since $Q>0$, it is immediate to obtain $X = \sqrt{Q}$. Moreover, from invertibility of $Q$ we infer invertibility of $X$, therefore we can set $R = -{X^*}^{-1}A^*$ and conclude the proof.
\end{proof}

The previous Proposition shows that we can actually recover all Ornstein-Uhlenbeck semigroups
with non-degenerate diffusion matrix starting from a gaussian QMS. The case of a degenerate diffusion matrix is
more complicated and is beyond the scope of this work.

\section{The spectral gap for the GNS embedding} \label{sec:spectralGap}

In this section we explicitly compute the spectral gap for the $i_2$ embedding.
Even when not explicitly stated we will assume $Z$ stable and $\CZ >0$,
or equivalently that the semigroup has $2d$ Kraus' operators, thanks to Proposition \ref{prop:CandZ}.
We will prove in section \ref{sec:invariantStates} that the these two requirements are not really restrictive
since they are necessary conditions for the semigroup to have a spectral gap $g>0$.

\begin{remark}
In the following we will always assume $\zeta = 0$ to simplify calculations even though all the results remain true without this assumption.
\end{remark}

We will show the existence of a spectral gap by first finding a candidate $g>0$ and then proving it actually satisfies \eqref{def-g}.
We will work on a dense subspace of $\mathfrak{I}_2(\mathsf{h})$, the space of Hilbert-Schmidt operators on $\mathsf{h}$,
which is the range of the immersion $i_2$ of the space of linear combination of Weyl operators
\[
	\mathcal{W} = \left\{\sum_{j=1}^l \eta_j W(z_j) \mid
      l \geq 1, z_1, \dots, z_l \in \mathbb{C}^d, \eta_1, \dots, \eta_l \in \mathbb{C} \right\}.
\]
In the following, for any $x \in \mathcal{B}(\mathsf{h})$ we will denote by $\widetilde{x}$ the projection of $x\rho^{1/2}$ on
the orthogonal of $\rho^{1/2}$ in $\mathfrak{I}_2(\mathsf{h})$ given by
\begin{equation}\label{def-xtilde}
\widetilde{x}:=x\rho^{1/2}-\left\langle \rho^{1/2}, x\rho^{1/2}\right\rangle_2\rho^{1/2}.
\end{equation}

Note that for $x \in \mathcal{W}$ we have
$$\left\langle \rho^{1/2}, x\rho^{1/2}\right\rangle_2
=\sum_{j=1}^l\eta_j\,\tr\left(\rho W(z_j)\right) =\sum_{j=1}^l\eta_j\, \mathrm{e}^{-\frac{1}{2}\Re\langle z_j,S z_j\rangle},
$$
so that the orthogonalization $\widetilde{x}$ yields
\begin{equation} \label{eq:combinationWeylTilde}
\widetilde{x} =\sum_{j=1}^l\eta_j\left( W(z_j) - \mathrm{e}^{-\frac{1}{2}\Re\langle z_j,S z_j\rangle}\unit \right) \rho^{1/2}.
\end{equation}

We start with the following

\begin{lemma}\label{lem:normSpectralGap}
Let $x \in \mathcal{W}$ and set $\xi_j = \mathrm{e}^{-\frac{1}{2} \langle z_j,S z_j\rangle}\eta_j$ for all $j=1,\ldots,n$.
Then
\begin{equation}\label{eq:normSpectralGap1}
\left\Vert T_t(\tilde{x})\right\Vert_2^2 =
\sum_{j,k=1}^l\overline{\xi}_j\xi_k \left(
      \mathrm{exp} \scalar{\me^{t\mathbf{Z}}\begin{bmatrix}
	\Re z_j \\ \Im{z_j}
	\end{bmatrix}}{\left( \mathbf{S} + \mi \mathbf{J}\right)\me^{t\mathbf{Z}} \begin{bmatrix}
	\Re z_k \\ \Im{z_k}
	\end{bmatrix}} - 1\right)
\end{equation}
for all $t\geq 0$.
\end{lemma}

\begin{proof}
By formula (\ref{eq:combinationWeylTilde}) we have
\begin{align*}
 \left\Vert T_t(\widetilde{x})\right\Vert_2^2
 % & = \left\Vert T_t(x)-\langle\rho^{1/2},x\rangle\,\rho^{1/2}\right\Vert_2^2
 & = \left\Vert \sum_{j=1}^l\eta_j\left( \T_t(W(z_j)) - \mathrm{e}^{-\frac{1}{2}\Re\langle z_j,S z_j\rangle}\unit \right) \rho^{1/2}\right\Vert_2^2
 \\
 & = \sum_{j,k=1}^l\overline{\eta}_j\eta_k \tr\left(\rho \left(\mathcal{T}_t(W(-z_j))
   - \mathrm{e}^{-\frac{1}{2} \Re\langle z_j,S z_j\rangle}\unit\right)
   \left(\mathcal{T}_t(W(z_k))
   - \mathrm{e}^{-\frac{1}{2} \Re\langle z_k,S z_k\rangle}\unit\right)\right)
\end{align*}
which is equal to
\begin{align*}
   & \ \sum_{j,k=1}^l\overline{\eta}_j\eta_k \, \tr\left(\rho \left( \mathcal{T}_t(W(-z_j))\mathcal{T}_t(W(z_k))
   -  \mathrm{e}^{-\frac{1}{2} \Re\langle z_j,S z_j\rangle -\frac{1}{2} \Re\langle z_k,S z_k\rangle}
   \unit\right)\right) \\
 &= \sum_{j,k=1}^l\overline{\eta}_j\eta_k \Big[ \tr\left(\rho \,
      \mathrm{e}^{-\frac{1}{2}\Re\int_0^t \left( \left\langle \mathrm{e}^{sZ}z_j,C\mathrm{e}^{sZ}z_j\right\rangle
      + \left\langle \mathrm{e}^{sZ}z_k,C\mathrm{e}^{sZ}z_k\right\rangle\right)\mathrm{d}s}\,
      W(-\mathrm{e}^{tZ}z_j) W(\mathrm{e}^{tZ}z_k)  \right) \\
 &  \qquad\qquad
 -\, \mathrm{e}^{-\frac{1}{2} \Re\langle z_j,S z_j\rangle -\frac{1}{2} \Re\langle z_k,S z_k\rangle}\Big].
\end{align*}

Keeping into account the CCR relation for Weyl operators \eqref{eq:WeylCCR} we get
\begin{eqnarray}
 \left\Vert T_t(\widetilde{x})\right\Vert_2^2
  &=& \sum_{j,k}\overline{\eta}_j\eta_k \Big[
      \mathrm{e}^{-\frac{1}{2}\Re\int_0^t \left( \left\langle \mathrm{e}^{sZ}z_j,C\mathrm{e}^{sZ}z_j\right\rangle
      + \left\langle \mathrm{e}^{sZ}z_k,C\mathrm{e}^{sZ}z_k\right\rangle\right)\mathrm{d}s}\, \nonumber \\
      &&\qquad \cdot\, \mathrm{e}^{\mi\Im\langle\mathrm{e}^{tZ}z_j, \mathrm{e}^{tZ}z_k \rangle
      -\frac{1}{2}\Re\left\langle \mathrm{e}^{tZ}(z_k-z_j),S\mathrm{e}^{tZ}(z_k-z_j) \right\rangle } \nonumber \\
 && \qquad - \mathrm{e}^{-\frac{1}{2} \Re\langle z_j,S z_j\rangle -\frac{1}{2} \Re\langle z_k,S z_k\rangle}\Big]. \label{eq:zj-zk}
\end{eqnarray}
Now note that using the expression \eqref{cond-su-S} for $S$ we have
\[
\Re\left\langle \mathrm{e}^{tZ}(z_k-z_j),S\mathrm{e}^{tZ}(z_k-z_j) \right\rangle
= \int_t^\infty \Re\left\langle \mathrm{e}^{sZ}(z_k-z_j),C\mathrm{e}^{sZ}(z_k-z_j) \right\rangle \mathrm{d}s
\]
for all $j,k$, which allows us to rewrite
\begin{align*}
 &   -\frac{1}{2}\Re\int_0^t \left( \left\langle \mathrm{e}^{sZ}z_j,C\mathrm{e}^{sZ}z_j\right\rangle
      + \left\langle \mathrm{e}^{sZ}z_k,C\mathrm{e}^{sZ}z_k\right\rangle\right)\mathrm{d}s
   + \mi\Im\langle\mathrm{e}^{tZ}z_j, \mathrm{e}^{tZ}z_k \rangle \\
    &  -\frac{1}{2}\Re\left\langle \mathrm{e}^{tZ}(z_k-z_j),S\mathrm{e}^{tZ}(z_k-z_j) \right\rangle \\
   &  =  -\frac{1}{2} \Re\langle z_j,S z_j\rangle -\frac{1}{2} \Re\langle z_k,S z_k\rangle
   + \Re\langle \mathrm{e}^{tZ}z_j, S \mathrm{e}^{tZ}z_k \rangle+\mi\Im\langle\mathrm{e}^{tZ}z_j, \mathrm{e}^{tZ}z_k \rangle.
\end{align*}
Summing up, we can write $\left\Vert T_t(\widetilde{x})\right\Vert_2^2$ as
\begin{align*}
 \sum_{j,k}\overline{\eta}_j\eta_k \Big[
      &\mathrm{e}^{-\frac{1}{2} \Re\langle z_j,S z_j\rangle -\frac{1}{2} \Re\langle z_k,S z_k\rangle
   + \Re\langle \mathrm{e}^{tZ}z_j, S \mathrm{e}^{tZ}z_k \rangle
   +\mi\Im\langle\mathrm{e}^{tZ}z_j, \mathrm{e}^{tZ}z_k \rangle } \\
   &  - \mathrm{e}^{-\frac{1}{2} \Re\langle z_j,S z_j\rangle -\frac{1}{2} \Re\langle z_k,S z_k\rangle}\Big].
\end{align*}
Letting
\(
\xi_j = \mathrm{e}^{-\frac{1}{2} \Re\langle z_j,S z_j\rangle}\eta_j
\)
for all $j=1,\ldots,n$, we obtain
\begin{align*}
 \left\Vert T_t(\widetilde{x})\right\Vert_2^2
  & =  \sum_{j,k=1}^l\overline{\xi}_j\xi_k \left(
      \mathrm{e}^{\Re\langle \mathrm{e}^{tZ}z_j, S \mathrm{e}^{tZ}z_k \rangle
   +\mi\Im\langle\mathrm{e}^{tZ}z_j, \mathrm{e}^{tZ}z_k \rangle }
   - 1\right) \\
 \left\Vert \widetilde{x}\right\Vert_2^2
 & =  \sum_{j,k=1}^l\overline{\xi}_j\xi_k \left(
 \mathrm{e}^{ \Re\langle z_j, S z_k\rangle + \mi \Im\langle z_j, z_k \rangle}
 -1\right).
\end{align*}

Finally, recalling equation \eqref{legame-prod-scal}, we can write
\begin{equation}\label{complex}
	\Re\langle \mathrm{e}^{tZ}z_j, S \mathrm{e}^{tZ}z_k \rangle
   +\mi\Im\langle\mathrm{e}^{tZ}z_j, \mathrm{e}^{tZ}z_k \rangle = \scalar{\me^{t\mathbf{Z}}\begin{bmatrix}
	\Re z_j \\ \Im{z_j}
	\end{bmatrix}}{\left( \mathbf{S} + \mi \mathbf{J}\right)\me^{t\mathbf{Z}} \begin{bmatrix}
	\Re z_k \\ \Im{z_k}
	\end{bmatrix}}.
\end{equation}
This completes the proof.
\end{proof}

\medskip
The previous lemma allows us to rewrite equation \eqref{def-g} for $\tilde{x}$ with $x \in \mathcal{W}$.
In this linear space of operators the condition for the spectral gap reduces to positive definiteness of a kernel
as highlighted in the following proposition.

\begin{proposition} \label{prop:spectralGapLinearCombinations}
For all $t \geq 0 $,  defining
\[
	s_t(z,w) := \scalar{\me^{t\mathbf{Z}}\begin{bmatrix}
	\Re z \\ \Im{z}
	\end{bmatrix}}{\left( \mathbf{S} + \mi \mathbf{J} \right) \me^{t\mathbf{Z}} \begin{bmatrix}
	\Re w \\ \Im{w}
	\end{bmatrix}}\qquad\forall\,t\geq 0,\ z,w\in\mathbb{C}^d
\]
for all $x \in \mathcal{W}$, setting $\xi_j = \me^{-\frac{1}{2}\scalar{z_j}{Sz_j}}\eta_j$ as in Lemma \ref{lem:normSpectralGap}, we have
\begin{equation}\label{eq:Ttxst}
\left\Vert T_t(\tilde{x})\right\Vert_2^2 = \sum_{n \geq 1}\sum_{j,k=1}^l  \frac{\conj{\xi_j}\xi_k}{n!} \left( s_t (z_j, z_k) \right)^n.
\end{equation}

In particular, if the kernel
	\[
	K_{n,t} (z,w ) = \me^{-2gt}\left(s_0(z,w)\right)^n - \left(s_t(z,w)\right)^n, \quad z,w \in \mathbb{C}^d,
\]
is positive definite for every $n \geq 1$, the spectral gap inequality \eqref{def-g} holds for any $\tilde{x}$ with $x \in \mathcal{W}$.
\end{proposition}
\begin{proof}
The power series expansion of the exponential immediately yields (\ref{eq:Ttxst}) from (\ref{eq:normSpectralGap1}).
This formula for $t>0$ and for $t=0$ allows us to rewrite \eqref{def-g} as
	\begin{align*}
	0  \leq \me^{-2gt}\norm{\tilde{x}}_2^2 - \norm{T_t(\tilde{x})}_2^2
& = \sum_{n \geq 1}\frac{1}{n!}\sum_{j,k=1}^l \conj{\xi_j} \xi_k
\left( \me^{-2gt} \left(s_0(z_j,z_k)\right)^n - \left( s_t(z_j, z_k)\right)^n\right),
\end{align*}
which proves the claim.
\end{proof}

\medskip

The characterization of Proposition \ref{prop:spectralGapLinearCombinations} can be exploited  to
find a suitable candidate for the spectral gap. We begin by the following Lemma.

\begin{lemma} \label{lem:ineqBound}
	Let $Y \in M_{n}(\mathbb{C})$, $n\geq 1$, and let $\omega_0$ be the maximum eigenvalue of $Y+Y^*$. Then
	\begin{equation} \label{eq:ineqBound}
		\norm{\me^{tY}v}^2 \leq \me^{t\omega_0}\norm{v}^2 \quad \forall\, t \geq 0,\ v \in \mathbb{C}^n.
	\end{equation}
	Moreover $\omega_0$ is the optimal choice in the inequality, i.e.
	\[
		\omega_0 = \min \{ \omega\in\mathbb{R}\,:\, \norm{\me^{tY}v}^2 \leq \me^{t\omega}\norm{v}^2 \forall\, t \geq 0,\ v \in \mathbb{C}^n \}.
	\]
\end{lemma}
\begin{proof}
Differentiating the function
	\(
		f_v(t) = \norm{\me^{tZ}v}^2
	\)
	we find
	\begin{equation} \label{eq:diffEq}
		\frac{\md}{\md t} f_v(t) = \scalar{ \me^{tY} z}{(Y+Y^*)\me^{tY} z} \leq \omega_0 f_v(t)
	\end{equation}
and Gronwall's Lemma implies \eqref{eq:ineqBound}. On the other hand, if we choose $\omega < \omega_0$
and we consider the eigenvector $v$ of $Y+Y^*$ associated with the eigenvalue $\omega_0$, then
	\[
		\left. \frac{\md}{\md t} \me^{-t\omega}f_v(t) \right|_{t=0} = (\omega_0 - \omega) \norm{v}^2 > 0.
	\]
	In particular we have $\me^{-t\omega}f_v(t) > \norm{v}^2$ in a right neighbourhood of $t=0$ which means $\omega_0$ is the optimal bound in inequality \eqref{eq:ineqBound}.
\end{proof}

\medskip

Let $\tilde{\mathbf{S}}= \mathbf{S}+\mi \mathbf{J}$.
Proposition \ref{prop:spectralGapLinearCombinations} with  $n=1$ and $x=W(z)$ yields
\begin{equation}\label{det-g}
K_{1,t}(z,z)=\me^{-2gt} s_0(z,z) - s_t(z, z)=\me^{-2gt}\norm{\tilde{\mathbf{S}}^{1/2}\begin{bmatrix}
	\Re z \\ \Im z
	\end{bmatrix}}^2-\norm{\me^{t \tilde{\mathbf{S}}^{1/2}\mathbf{Z}\tilde{\mathbf{S}}^{-1/2}}\tilde{\mathbf{S}}^{1/2}\begin{bmatrix}
	\Re z \\ \Im z
	\end{bmatrix}}^2
\end{equation}
for all $t\geq0$, and we get a natural candidate for $2g$ from  Lemma \ref{lem:ineqBound} applied to $Y=\tilde{\mathbf{S}}^{1/2}\mathbf{Z}\tilde{\mathbf{S}}^{-1/2}$ as a $2d \times 2d$ matrix.

\begin{proposition} \label{prop:omega0Properties}
Suppose that $Z$ is stable and $\CZ>0$. Let $\omega_0$ be the greatest eigenvalue of the $2d\times 2d$
 matrix $\tilde{\mathbf{S}}^\frac{1}{2}\mathbf{Z} \tilde{\mathbf{S}}^{-\frac{1}{2}}
 + \tilde{\mathbf{S}}^{-\frac{1}{2}}\mathbf{Z}^* \tilde{\mathbf{S}}^{\frac{1}{2}}$, then
$-\omega_0$ is the smallest eigenvalue of $\mathbf{C_{\mathbf{Z}}\tilde{S}}^{-1}$ and
$\omega_0 <0$.
\end{proposition}
\begin{proof}
The condition on the covariance operator $\mathbf{S}$ of an invariant state $\mathbf{Z}^* \mathbf{S} + \mathbf{SZ} = -\mathbf{C}$
immediately yields
\[
\mathbf{Z}^*\tilde{\mathbf{S}}+\tilde{\mathbf{S}}\mathbf{Z}=-\mathbf{C}+\mi(\mathbf{Z}^*\mathbf{J}+\mathbf{J}\mathbf{Z})=
-\mathbf{C}_{\mathbf{Z}}.
\]
Left and right multiplying by $\tilde{\mathbf{S}}^{-\frac{1}{2}}$ we get
\begin{equation}\label{itm:dissipativeGenerator}
\tilde{\mathbf{S}}^\frac{1}{2}\mathbf{Z} \tilde{\mathbf{S}}^{-\frac{1}{2}} + \tilde{\mathbf{S}}^{-\frac{1}{2}}\mathbf{Z}^* \tilde{\mathbf{S}}^{\frac{1}{2}} = - \tilde{\mathbf{S}}^{-\frac{1}{2}}\mathbf{C}_{\mathbf{Z}}\tilde{\mathbf{S}}^{-\frac{1}{2}}.
\end{equation}
Since $\mathbf{C}_{\mathbf{Z}}>0$ by our assumptions, we get $\omega_0 < 0$. On the other hand,  $\tilde{\mathbf{S}}^{-\frac{1}{2}}\mathbf{C}_{\mathbf{Z}}\tilde{\mathbf{S}}^{-\frac{1}{2}}$ is similar to $\mathbf{C}_{\mathbf{Z}}\tilde{\mathbf{S}}^{-1}$ and then $-\omega_0$ is the smallest eigenvalue
of $\mathbf{C}_{\mathbf{Z}}\tilde{\mathbf{S}}^{-1}$.
	\end{proof}
\medskip

Our goal now is to prove that $g=-\omega_0/2$ is actually the spectral gap.
We begin by showing that equation \eqref{def-g} holds for $\tilde{x}$ with $x\in \mathcal{W}$, i.e.
that $K_{n,t}$ is a definite positive kernel.
As a preliminary result we prove the following
\begin{lemma} \label{lem:kernelPosDef}
Let $n \geq 1$ and $t \geq 0$. The kernel
\begin{equation} \label{eq:kernelRoot}
	K^\prime_{n,t} (z,w ) = \me^{t\omega_0/n}s_0(z,w) - s_t(z,w), \quad z,w \in \mathbb{C}^d
\end{equation}
is positive definite.
\end{lemma}
\begin{proof}
Note that we just need to prove the proposition for $n = 1$ because we have
\[
	\sum_{j,k} \conj{c_j}c_k K^\prime_{n,t}(z_j,z_k) = \sum_{j,k} \conj{c_j}c_k\left( K^\prime_{1,t}(z_j,z_k) + \left( \me^{-2tg/n} - \me^{-2tg} \right) s_0(z_j,z_k)\right),
\]
($c_j \in \mathbb{C}, z_j \in \mathbb{C}^d$) which is positive if $K^\prime_{1,t}$ is positive definite,
since $s_0$ is positive definite and $\me^{-2tg/n} - \me^{-2tg} >0$.
Therefore, fix $n=1$ and let us prove positive definiteness of \eqref{eq:kernelRoot}. By sesquilinearity of $s_t$ we have
\[
	\sum_{j,k} \conj{c_j}c_k K^\prime_{1,t}(z_j,z_k) = \me^{-2tg}s_0\left( \xi, \xi\right) - s_t(\xi, \xi),
\]
where $\xi = \sum_j c_j z_j$. Since this quantity is positive by our choice of $2g$ and
Lemma \ref{lem:ineqBound},
$K^\prime_{1,t}$ is positive definite and the proof is complete.
\end{proof}

\smallskip

We are now in a position to prove the key inequality
\begin{proposition} \label{prop:spectralGapUpperBound}
	For all $x \in \mathcal{W}$ and $t\geq 0$ we have
	\[
		\norm{T_t(\tilde{x})}_2^2 \leq \me^{-2gt} \norm{\tilde{x}}_2^2\, .
	\]
\end{proposition}
\begin{proof}
Thanks to Proposition \ref{prop:spectralGapLinearCombinations} it suffices to show that the kernel
\[
	K_{n,t} (z,w ) = \me^{t\omega_0}\left(s_0(z,w)\right)^n - \left(s_t(z,w)\right)^n, \quad z,w \in \mathbb{C}^d,
\]
is positive definite.
Recalling the kernel defined in \eqref{eq:kernelRoot}, we have
\begin{align*}
	K_{n,t} (z,w) = K^\prime_{n,t}(z,w) \sum_{m=0}^{n-1}\me^{\frac{n-1-m}{n}\,t\omega_0}\left(s_0(z,w)\right)^{n-1-m}\left(s_t(z,w)\right)^{m}.
\end{align*}
Since both $s_t$ and $s_0$ are positive definite kernels, their powers and their linear combinations with positive coefficients are still positive definite by \cite{Partha} Corollary 15.2, and so is $K_{n,t}$.
\end{proof}
\smallskip

\smallskip
Finally we prove that the bound for the spectral gap given by Proposition \ref{prop:spectralGapUpperBound} is actually optimal,
even if we just consider Hilbert-Schmidt operators $\tilde{x}$ with $x \in \mathcal{W}$.

\begin{proposition} \label{prop:spectralGapSharp}
	Let $\omega < \omega_0$. There exist $\delta > 0$ and ${x} \in\mathcal{W}$ such that
	\[
		\norm{T_t(\tilde{x})}_2^2 > \me^{t\omega} \norm{\tilde{x}}_2^2\ \quad \forall\, t \in (0, \delta).
	\]
\end{proposition}
\begin{proof}
	Let $z \in \mathbb{C}^{2d}$ such that $\tilde{\mathbf{S}}^\frac{1}{2}z$ is an eigenvector of
$\tilde{\mathbf{S}}^\frac{1}{2}\mathbf{Z} \tilde{\mathbf{S}}^{-\frac{1}{2}}
+ \tilde{\mathbf{S}}^{-\frac{1}{2}}\mathbf{Z}^* \tilde{\mathbf{S}}^{\frac{1}{2}} $ associated with the eigenvalue $\omega_0$.
Write $z=[z_p,z_q]^T$ with $z_p,z_q\in\mathbb{C}^d$ and consider  $z_1, z_2 \in \mathbb{C}^d$ defined by $z_1=\Re z_p +\mi \Re z_q$, $z_2=\Im z_p +\mi \Im z_q$,  so that $[\Re z_1, \Im z_1]^T = \Re z$ and $[\Re z_2, \Im z_2]^T= \Im z$.
For all $r\in\mathbb{R}$, considering $x_r = W(rz_1) + \mi W(rz_2)=\sum_{j=1}^2 c_j W(rz_j)$, with $c_1 = 1, c_2=\mi$, we have
	\[
		\widetilde{x_r} = \sum_{j=1}^2c_j\left( W(rz_j) - \me^{-\frac{r^2}{2}s_0(z_j,z_j)} \unit\right)\rho^{1/2}
	\]
	by equation \eqref{eq:combinationWeylTilde}.
	 Setting again $\xi_{j,r} = \me^{-\frac{r^2}{2} s_0(z_j,z_j)}c_j$, using equation \eqref{eq:normSpectralGap1}
and recalling the definition of $s_t(z,w)$ in Proposition \ref{prop:spectralGapLinearCombinations} we then have
\[
\left\Vert T_t(\tilde{x})\right\Vert_2^2
= \sum_{j,k=1}^l\overline{\xi}_{j,r}\,\xi_{k,r} \left(\mathrm{e}^{r^2 s_t(z_j,z_k)}-1 \right)
\]
Taking the derivative at $t=0$ we find
	\[
	\left.\frac{\md}{\md t} s_t(z_j,z_k)\right|_{t=0}=\scalar{\tilde{\mathbf{S}}^{\frac{1}{2}}\begin{bmatrix}\Re z_j\\ \Im z_j \end{bmatrix}}{\left(\tilde{\mathbf{S}}^{1/2}\mathbf{Z}\tilde{\mathbf{S}}^{-1/2}
+\tilde{\mathbf{S}}^{-1/2}\mathbf{Z}^*\tilde{\mathbf{S}}^{1/2}\right)\tilde{\mathbf{S}}^{\frac{1}{2}}
\begin{bmatrix}\Re z_k\\ \Im z_k \end{bmatrix} }
	\]
	Consider now
	\begin{align*}
		f(r) &:= \left.\frac{\md}{\md t} \norm{T_t(\tilde{x_r})}_2^2 - \me^{t\omega} \norm{\tilde{x_r}}_2^2 \right|_{t=0} \\
		&= \sum_{j,k=1}^2 \overline{\xi}_{j,r}\,\xi_{k,r}
\left(\left(r^2\left.\frac{\md}{\md t} s_t(z_j,z_k)\right|_{t=0} -\omega \right)\me^{r^2s_0(z_j,z_k)} + \omega \right)
	\end{align*}
	It holds $f(0) = f^\prime(0) = 0$ while
	\[
		f^{\prime\prime}(0) = 2\sum_{j,k=1}^2 \conj{c_j}c_k \left(\left.\frac{\md}{\md t} s_t(z_j,z_k)\right|_{t=0} -\omega s_0(z_j,z_k) \right)	 = 2 (\omega_0 - \omega) s_0(z,z)
	\]
	where the last equality follows from sesquilinearity of the derivative of $s_t$ and $s_0$ as well as the fact that $\tilde{\mathbf{S}}^\frac{1}{2}z$ is an eigenvector. In particular, recalling that $\omega<\omega_0<0$, we find
$f^{\prime\prime}(0)>0$ because the invariant state is faithful
and $z\neq 0$. This implies that there exists $\delta_0>0$ such that $f(r) >0$ for every $r \in (0, \delta_0)$.
This implies in turn that for every $r \in (0, \delta_0)$ there exists a $\delta > 0$ such that
	\[
		\norm{T_t(\tilde{x_r})}_2^2 > \me^{t\omega} \norm{\tilde{x_r}}_2^2, \quad \forall t \in (0, \delta).
	\]
\end{proof}

\smallskip

Proposition \ref{prop:spectralGapUpperBound} proves the inequality \eqref{def-g} for the spectral gap
for elements $\tilde{x}$ with $x \in \mathcal{W}$.
The following Theorem shows that it holds for all the remaining elements.

\begin{theorem} \label{thm:spectralGap}
	Let $\T$ be a gaussian QMS with $Z$ stable, $\CZ >0$ and $\zeta=0$. Then $\T$ has a spectral gap $g=-\omega_0/2 > 0$.
\end{theorem}
\begin{proof}
By Proposition \ref{prop:spectralGapUpperBound} we have $\norm{T_t(\tilde{x})}_2^2 \leq \me^{-2gt} \norm{\tilde{x}}_2^2$
for all  $t \geq 0$ and all $\tilde{x}\in\mathfrak{I}_2(\mathsf{h})$ written in the form \eqref{eq:combinationWeylTilde}.
To show that it also holds for any element of $\mathfrak{I}_2(\mathsf{h})$ orthogonal to $\rho^{1/2}$ we show that the set of those $\tilde{x}$ is dense in the orthogonal subspace of
$\rho^{1/2}$. To this end, consider $y\in \mathfrak{I}_2(\mathsf{h})$ orthogonal to $\rho^{1/2}$ such that
$\operatorname{tr}\left(y^*\tilde{x}\right)=0$ for all $\tilde{x}= x\rho^{1/2}-\scalar{\rho^{1/2}}{x\rho^{1/2}}_2\rho^{1/2}$ with $x\in\mathcal{W}$. In particular, taking $x=W(z)$ we have
\[
0=\operatorname{tr}\left(y^*\tilde{x}\right)
=\operatorname{tr}\left(y^*\left(W(z)-\operatorname{tr}\left(\rho\,W(z)\right)\unit\right)\rho^{1/2}\right)=\operatorname{tr}\left(\rho^{1/2}y^*W(z)\right)
\]
for all $z\in\mathbb{C}^d$. Since $\rho^{1/2}y^*$ is a trace class operator, the weak$^*$ density of Weyl operators in $\mathcal{B}(\mathsf{h})$
implies $\rho^{1/2}y^*=0$, namely $y=0$ because $\rho$ is faithful.

Proposition \ref{prop:spectralGapSharp} now concludes the proof, showing that $g$ is the optimal choice.
\end{proof}

\section{Conditions on the invariant state} \label{sec:invariantStates}

In this section we want to discuss the assumptions we made on $Z$ and $\CZ$ to show that they are not restrictive when looking for a positive spectral gap. We have indeed the following result.

\begin{theorem} \label{thm:assumptionsJustification}
	Let $\mathcal{T}$ be a gaussian QMS and suppose it has a spectral gap $g$. Then $Z$ is stable and $\CZ > 0$.
\end{theorem}

We will split the proof in two steps. We start by showing the stability of $Z$, which is due to the requirement
of the existence of an invariant state and the asymptotic stability-like requirement in \eqref{def-g}.

\begin{lemma} \label{lem:stabilityZ}
	If $Z$ is not stable then there exists $z \in \mathbb{C}^d$ that satisfies one of the following:
	\begin{enumerate}
		\item $\me^{tZ}z \in \ker C$ for every $t \geq 0$;
		\item $\lim_{t \to \infty} \int_0^t \Re\scalar{\me^{sZ}z}{C\me^{sZ}z} \md s = +\infty.$
	\end{enumerate}
\end{lemma}
\begin{proof}
	Suppose $\lambda$ is an eigenvalue for $\mathbf{Z}$ with $\Re\lambda \geq 0$ and let $w \in \mathbb{C}^{2d}$ be its associated eigenvector. This implies that $\conj{w}$ is an eigenvector for $\conj{\lambda}$. Consider $x = w + \conj{w}$ and $y=\mi(w-\conj{w})$, both belonging to $\mathbb{R}^{2d}$ and satisfying
	\[
		\mathbf{Z}x = \Re\lambda x +\Im\lambda y, \quad \mathbf{Z}y = -\Im\lambda x + \Re\lambda y.
	\]
	If $\Im \lambda = 0$ we can set $z\in \mathbb{C}^d$ as the element corresponding to $w$ and have $\me^{tZ}z = \me^{t\Re\lambda}z$, in particular
	\[
		\int_0^t \Re\scalar{\me^{sZ}z}{C\me^{sZ}z} \md s = \Re\scalar{z}{Cz} \int_0^t \me^{2s\Re\lambda} \md s.
	\]
	This implies either $z \in \ker C$ or it converges to infinity, proving the result in the case $\Im\lambda = 0$. Suppose then $\Im\lambda \neq 0$. On the space generated by $x,y$ we have
	\[
		Z = \begin{bmatrix}
			\Re\lambda & \Im\lambda \\
			-\Im\lambda & \Re\lambda
		\end{bmatrix}, \quad \me^{tZ} = \me^{t\Re\lambda}\begin{bmatrix}
			\cos(t\Im\lambda) & \sin(t\Im\lambda) \\
			-\sin(t\Im\lambda) & \cos(t\Im\lambda)
		\end{bmatrix}.
	\]
	For any $z \in \mathbb{C}^d$ corresponding to a real linear combination of $x, y$ we have
	\begin{align*}
		\int_0^t &\Re\scalar{\me^{sZ}z}{C\me^{sZ}z} \md s\\
		& = \int_0^t  \scalar{\begin{bmatrix}
			\cos(s\Im\lambda) & \sin(s\Im\lambda) \\
			-\sin(s\Im\lambda) & \cos(s\Im\lambda)
		\end{bmatrix}\begin{bmatrix}
			\Re z \\ \Im z
		\end{bmatrix}}{\mathbf{C} \begin{bmatrix}
			\cos(s\Im\lambda) & \sin(s\Im\lambda) \\
			-\sin(s\Im\lambda) & \cos(s\Im\lambda)
		\end{bmatrix}\begin{bmatrix}
			\Re z \\ \Im z
		\end{bmatrix}} \me^{2s\Re\lambda} \md s \\
		&\geq \int_0^t  \scalar{\begin{bmatrix}
			\cos(s\Im\lambda) & \sin(s\Im\lambda) \\
			-\sin(s\Im\lambda) & \cos(s\Im\lambda)
		\end{bmatrix}\begin{bmatrix}
			\Re z \\ \Im z
		\end{bmatrix}}{\mathbf{C} \begin{bmatrix}
			\cos(s\Im\lambda) & \sin(s\Im\lambda) \\
			-\sin(s\Im\lambda) & \cos(s\Im\lambda)
		\end{bmatrix}\begin{bmatrix}
			\Re z \\ \Im z
		\end{bmatrix}} \md s
	\end{align*}
	The scalar product inside the integral is non-negative and periodic as a function of $s$, therefore either it is identically zero or the integral diverges. Eventually, the scalar product being identically zero is equivalent to $\me^{tZ}z \in \ker C$ for every $t\geq 0$, completing the proof.
\end{proof}

\begin{proposition} \label{prop:gapImpliesZstable}
	Let $\mathcal{T}$ be a gaussian QMS with an invariant state $\rho$ and a positive spectral gap $g$. Then $Z$ is stable.
\end{proposition}
\begin{proof}
	Suppose $\mathbf{Z}$ has an eigenvalue $\lambda$ with $\Re\lambda \geq 0$. From Lemma \ref{lem:stabilityZ} we can find $z \in \mathbb{C}^d$ such that either  $\lim_{t \to \infty} \int_0^t \Re\scalar{\me^{sZ}z}{C\me^{sZ}z} \md s = +\infty$ or $\me^{tZ}z \in \ker C$ for every $t \geq 0$. In the former case, the constant $c_t(z)$ multiplying the Weyl operator in \eqref{eq:explWeyl} satisfies
	\begin{equation}\label{eq:ctz}
		c_t(z) = \exp \left( -\frac{1}{2} \int_0^t \rescalar{\me^{sZ}z}{C\me^{sZ}z} \md s \right) \longrightarrow 0.
	\end{equation}

We can explicitly compute	
	
	\begin{equation} \label{eq:charFunctAt0}
		|\hat{\rho}(z)| = |\tr ( \rho W(z)| = |\tr (\rho \mathcal{T}_t (W(z))| = |c_t(z) \hat{\rho}(\me^{tZ}z)| = c_t(z) |\hat{\rho}(\me^{\lambda t}z)| \leq c_t(z),
	\end{equation}
	
and,, for every $r>0$, equation \eqref{eq:charFunctAt0} reads
	\(
		|\hat{\rho}(rz)| = 0.
	\)
	This is a contradiction since the characteristic function is continuous and $\hat{\rho}(0) = 1$.
	On the other hand, if $z \in \ker C$, from explicit computations
	\[
		\identity = \mathcal{T}_t(W(z)) \mathcal{T}_t(W(z))^*.
	\]
	Therefore
	\[
		\norm{T_t(W(z)\rho^{1/2} )}^2_2 = \tr \left( \rho \mathcal{T}_t(W(z))^* \mathcal{T}_t(W(z)) \right) = 1
	\]
	and
	\[
		\norm{T_t(W(z)\rho^{1/2}) - \scalar{\rho^{1/2}}{W(z)\rho^{1/2}}_2 \rho^{1/2}}^2_2 \geq 1 - \modulo{\hat{\rho}(z)}^2.
	\]
	For the semigroup to have a spectral gap and inequality \eqref{def-g} to hold we must have $\modulo{\hat{\rho}(z)} = 1$ which is a contradiction since then
	\[
		\norm{W(z)\rho^{1/2} - \scalar{\rho^{1/2}}{W(z)\rho^{1/2}}_2 \rho^{1/2}}_2^2 = 1 - \modulo{\hat{\rho}(z)}^2 = 0.
	\]
\end{proof}

To prove the necessity of $\CZ > 0$, by contradiction, one can show that if $\ker \CZ \neq \{ 0 \}$ then \eqref{def-g}
does not hold for some $x \in \mathfrak{I}_2(\mathsf{h})$ and $t > 0$ and so there is no strictly positive spectral gap.
The intuition behind the proof is that if $(z_p,z_q) \in \ker \CZ$ with $z_p, z_q \in \mathbb{C}^d$ then \eqref{def-g}
fails when considering $x = \sum_{j=1}^d(z_{p,j} p_j - z_{q,j} q_j)\rho^{1/2}$. This choice for $x$ however
is not rigorous since $\sum_{j=1}^d(z_{p,j} p_j - z_{q,j} q_j) \not\in \mathcal{B}(\mathsf{h})$ and we have to work
out computations through a limiting procedure.
However, we can prove the following proposition whose technical details are deferred to Appendix B.

\begin{proposition} \label{prop:gapImpliesCZinvertible}
	Let $\mathcal{T}$ be a gaussian QMS with a faithful invariant state $\rho$. If $\ker \CZ \neq \{ 0\}$.
Then the semigroup has no spectral gap.
\end{proposition}

%\textcolor{blue}{ We defer the technical details to Appendix B and proceed to the
%proof of Theorem \ref{thm:assumptionsJustification}. }

\begin{proof}[Proof of Theorem \ref{thm:assumptionsJustification}]
	From Proposition \ref{prop:gapImpliesZstable} we have stability of $Z$, while Proposition \ref{prop:gapImpliesCZinvertible} gives $\CZ >0$.
\end{proof}

\section{A one-dimensional case} \label{sec:oneDimCase}
The detailed analysis of a one-dimensional case with the explicit computation of $g$ in terms
of a few parameters clarifies why the condition $\CZ>0$ is necessary for $g>0$. Moreover, it shows that one can find
a quantum Ornstein-Uhlenbeck process with $g=0$ with a classical subprocess, obtained by restriction to an
abelian subalgebra, with strictly positive spectral gap.

Fix $d=1$. In order to further simplify the notation we consider
\[
L_1=\mu a,\qquad L_2=\lambda a^\dagger,\qquad H=\Omega\, a^\dagger a + \kappa(a^{\dagger\, 2}+a^2)/2
\] with
$0<\lambda<\mu$, $\Omega,\kappa\in\mathbb{R}$. Define $\gamma=(\mu^2-\lambda^2)/2$.
In this framework a faithful invariant state exists if and only if $\gamma^2+\Omega^2-\kappa^2>0$
(see \cite{AFP} Theorem 9 and Proposition 5). Moreover
\[
\mathbf{C}= \left[\begin{array}{cc} \mu^2+\lambda^2 & 0 \\ 0 & \mu^2+\lambda^2 \end{array}\right], \quad
\mathbf{Z}= \left[\begin{array}{cc} -\gamma & \kappa-\Omega \\ \kappa+\Omega & -\gamma \end{array}\right], \quad
\mathbf{C}_{\mathbf{Z}} = \left[\begin{array}{cc} \mu^2+\lambda^2 & 2\mi\gamma \\ -2\mi\gamma & \mu^2+\lambda^2 \end{array}\right]
\]
Solving $\mathbf{Z}^*\mathbf{S}+\mathbf{S}\mathbf{Z}+\mathbf{C}=0$ we find the covariance matrix of the unique invariant state
\begin{equation} \label{eq:Sonedim}
\mathbf{S}= \frac{\mu^2+\lambda^2}{2\gamma(\gamma^2+\Omega^2-\kappa^2)}
\left[\begin{array}{cc} \gamma^2+\Omega(\Omega+\kappa) & \kappa\gamma \\ \kappa\gamma & \gamma^2+\Omega(\Omega-\kappa) \end{array}\right]
\end{equation}
Since  $\tilde{\mathbf{S}}^\frac{1}{2}\mathbf{Z} \tilde{\mathbf{S}}^{-\frac{1}{2}} + \tilde{\mathbf{S}}^{-\frac{1}{2}}\mathbf{Z}^* \tilde{\mathbf{S}}^{\frac{1}{2}} = - \tilde{\mathbf{S}}^{-\frac{1}{2}}\mathbf{C}_{\mathbf{Z}}\tilde{\mathbf{S}}^{-\frac{1}{2}}$,
we have $\mathbf{C}_{\mathbf{Z}}\tilde{\mathbf{S}}^{-1}=-\mathbf{Z}^*-\tilde{\mathbf{S}}\mathbf{Z} \tilde{\mathbf{S}}^{-1}$,
therefore
\[
\operatorname{tr}\left(\mathbf{C}_{\mathbf{Z}}\tilde{\mathbf{S}}^{-1}\right)=4\gamma.
\]
In addition
\[
\operatorname{det}\left(\mathbf{C}_{\mathbf{Z}}\tilde{\mathbf{S}}^{-1}\right) =
\frac{\operatorname{det}\left(\mathbf{C}_{\mathbf{Z}}\right)}{\operatorname{det}\left(\tilde{\mathbf{S}}\right)}
= \frac{4\mu^2\lambda^2}{\operatorname{det}(\mathbf{S})-1},
\]
yielding eventually
\[
\operatorname{det}\left(\mathbf{C}_{\mathbf{Z}}\tilde{\mathbf{S}}^{-1}\right)
= \frac{4\mu^2\lambda^2}{\frac{(\gamma^2+\Omega^2)(\mu^2+\lambda^2)^2}{4\gamma^2(\gamma^2+\Omega^2-\kappa^2)}-1}
= \frac{4 \mu^2\lambda^2 \gamma^2(\gamma^2+\Omega^2-\kappa^2)}{\mu^2\lambda^2(\gamma^2+\Omega^2)+\gamma^2\kappa^2}
\]
Eigenvalues of $\CZ\tilde{\mathbf{S}}^{-1}$ are the roots of $r^2-\operatorname{tr}\left(\mathbf{C}_{\mathbf{Z}}\tilde{\mathbf{S}}^{-1}\right) r
+ {\operatorname{det}\left(\mathbf{C}_{\mathbf{Z}}\right)}/{\operatorname{det}\left(\tilde{\mathbf{S}}\right)}=0$
and the spectral gap is one-half of
\[
\gamma-\sqrt{\gamma^2-\frac{ \mu^2\lambda^2 \gamma^2(\gamma^2+\Omega^2-\kappa^2)}{\mu^2\lambda^2(\gamma^2+\Omega^2)+\gamma^2\kappa^2}}
= \gamma\left(1-\frac{|\kappa|(\mu^2+\lambda^2)}{2\sqrt{(\mu^2\lambda^2(\gamma^2+\Omega^2)+\gamma^2\kappa^2)}}\right)
\]

Note that, for $\lambda=0$ and $\kappa\not=0$, even the invariant state is faithful by
\[
\operatorname{det}\left(\tilde{\mathbf{S}}\right)=\operatorname{det}(\mathbf{S})-1
=\frac{\kappa^2}{4\gamma^2(\mu^4+\Omega^2)}>0\, ,
\]
but the spectral gap is zero because $\operatorname{det}(\CZ)=0$. If, in addition $\kappa = 2\Omega $ so that $H=\Omega (q^2 +\unit)/2$
\[
\mathcal{L}(f(q)) = \frac{\mu^2+\lambda^2}{4}f''(q) - \frac{\mu^2-\lambda^2}{2}qf'(q)
\]
which is the generator of a classical Ornstein-Uhlenbeck process which is also symmetric.
The density of the invariant measure (up to the normalization constant) is $\mathrm{e}^{-(\mu^2-\lambda^2)q^2/(\mu^2+\lambda^2)}$
and the spectral gap is $(\mu^2-\lambda^2)/2 >0$.

\section{The spectral gap for the KMS embedding}\label{sect:KMSembed}
As we already recalled in the introduction, the invariant density can induce different semigroups $T$ on the space
of Hilbert-Schmidt operators due to non-commutativity. The $i_{1/2,2}$ embedding has the advantage of producing
a semigroup $T$ of self-adjoint operators in more cases (see \cite{FFVU}), but at the cost of longer calculations
because of the computations for quantities of the kind of $\tr(\rho^{1/2}W(w)\rho^{1/2}W(z))$. These can be simplified by symplectic
diagonalization of $S$ but their explicit derivation is deferred to Appendix \ref{sec:appC}.

In this section we present the computation of the spectral gap when we consider the KMS embedding, highlighting the differences with the GNS case and the changes to be made for the calculations to go through. The spectral gap is again defined through  condition \eqref{def-g},
however the semigroup $T_t$ is
no more defined by $T_t(x\rho^{1/2}) = \mathcal{T}_t(x)\rho^{1/2}$ but instead by an adjusted version for the KMS-embedding

\[
	T_t ( \rho^{\frac{1}{4}} x \rho^{\frac{1}{4}}) = \rho^\frac{1}{4}\mathcal{T}_t(x)\rho^{\frac{1}{4}},
\quad \forall x \in \mathcal{B}(\mathbf{h}).
\]
We will show that many results obtained in Section \ref{sec:spectralGap} follow similarly with this new embedding,
apart for some adjustments and many more calculations. The overall result can be briefly summarized by saying that the spectral
gap is no longer $-1/2$ times the greatest eigenvalue of
$\mathbf{\tilde{S}}^\frac{1}{2}\mathbf{Z}\mathbf{\tilde{S}}^{-\frac{1}{2}} + \mathbf{\tilde{S}}^{-\frac{1}{2}}\mathbf{Z}^*\mathbf{\tilde{S}}^\frac{1}{2}$
but $-1/2$ times the greatest eigenvalue of an alike matrix obtained by replacing $\mathbf{\tilde{S}}$ by another operator
$\breve{S}$, still linked with $S$. Remarkably, there will be no need to consider complexifications of matrices since,
in this case, real linear operators will be sufficient to describe the behaviour of the semigroup.

We start by considering $x \in \mathcal{W}$, this time we denote with $\breve{x}$ the projection of $\rho^\frac{1}{4}x\rho^{\frac{1}{4}}$ onto the orthogonal of $\rho^\frac{1}{2}$, namely
\begin{equation} \label{eq:brevex}
	\breve{x} = \rho^\frac{1}{4} \left[\sum_{j=1}^l \eta_j\left(W(z_j) - \me^{-\frac{1}{2}\Re\scalar{z_j}{Sz_j}}\identity\right)\right]\rho^{\frac{1}{4}}.
\end{equation}
An analogous result to Lemma \ref{lem:normSpectralGap} and Proposition \ref{prop:spectralGapLinearCombinations} holds, with some slight adjustments. First of all we recall that every covariance matrix of a gaussian state can be symplectically diagonalized \cite{KRPSymm}. Explicitly, there exists a symplectic transformation $M$ on $\mathbb{C}^d$, i.e. a real linear operator satisfying $\Im\scalar{Mz}{Mw} = \Im\scalar{z}{w}$ for every $z,w \in \mathbb{C}^d$, such that $S = M^TD_\sigma M$ with $D_\sigma$ the diagonal matrix with entries $\sigma_1, \dots, \sigma_d \in (1, +\infty)$. We can now state the following result.

\begin{lemma}\label{lem:normSpectralGapKMS}
Let $x \in \mathcal{W}$ and set $\xi_j = \mathrm{e}^{-\frac{1}{2} \langle z_j,S z_j\rangle}\eta_j$ for all $j=1,\ldots,n$.
Then
\begin{align}\label{eq:normSpectralGap1KMS}
\left\Vert T_t(\breve{x})\right\Vert_2^2 &=
\sum_{j,k=1}^l\overline{\xi}_j\xi_k \left(
      \mathrm{exp} \left(\Re\scalar{\me^{tZ}z_j}{\breve{S}\me^{tZ}z_k }\right) - 1\right)
\\                            \label{eq:normSpectralGap2KMS}
& = \sum_{j,k=1}^l  \conj{\xi_j}\xi_k \left(\mathrm{e}^{\breve{s}_t (z_j, z_k)}-1\right)= \sum_{n \geq 1}\sum_{j,k=1}^l  \frac{\conj{\xi_j}\xi_k}{n!} \left( \breve{s}_t (z_j, z_k) \right)^n
\end{align}
for all $t\geq 0$, where $\breve{S} = M^T D_\nu M$, $M$ is the above symplectic transformation, $\nu_j = \operatorname{csch}\coth^{-1}(\sigma_j)$ and

\[
	\breve{s}_t(z,w) := \Re\scalar{\me^{tZ}z}{\breve{S}\me^{tZ}w }\qquad\forall\,t\geq 0,\ z,w\in\mathbb{C}^d.
\]
In particular we have
\begin{equation}\label{s_t-diag-KMS}	
\breve{s}_t(z,z) = \norm{\me^{t \breve{S}^{1/2}Z\breve{S}^{-1/2}}\breve{S}^{1/2}z}^2
\end{equation}
for all $t\geq 0$ and $z\in\mathbb{C}^d$.
\end{lemma}
\begin{proof}
We defer the proof to Appendix C.
\end{proof}

The previous Lemma joins in a single statement results similar to those of Lemma \ref{lem:normSpectralGap} and Proposition \ref{prop:spectralGapLinearCombinations}, simply replacing all instances of $\mathbf{\tilde{S}}$ with $\breve{S}$.The matrices $\breve{S}$ and $S$ are diagonalized by the same symplectic transformation $M$ and we know the explicit relation between their symplectic eigenvalues $\sigma_j, \nu_j$. Furthermore, from the condition $\sigma_j >1$, we infer $\nu_j>0$ showing that $\breve{S}$ is still a positive and invertible matrix, which were the relevant properties that also $\mathbf{\tilde{S}}$ enjoyed. Letting $\breve{\omega}_0$ the greatest eigenvalue of $\breve{S}^\frac{1}{2}Z\breve{S}^{-\frac{1}{2}} + \breve{S}^{-\frac{1}{2}}Z^\sharp\breve{S}^\frac{1}{2}$ and $\breve{g}= -\breve{\omega}_0/2$, we can now proceed to compute the spectral gap, in total analogy to what we did for the GNS embedding.

With these new definition a new version for Proposition \ref{prop:spectralGapUpperBound} holds, yielding a lower bound for the spectral gap with the KMS embedding, namely $2\breve{g}$.
\begin{proposition} \label{prop:KMSembedding}
	Let $x \in \mathcal{W}$. It holds
	\[
		\norm{T_t(\breve{x})}_2^2 \leq \me^{-2\breve{g}t} \norm{\breve{x}}_2^2, \quad \forall t \geq 0.
	\]
\end{proposition}
The proof follows precisely the lines of the one in Section \ref{sec:spectralGap} except that Lemma \ref{lem:kernelPosDef} and all the computations have to be performed with the $\breve{\cdot}$ counterparts of the quantities involved.

In the same way we can prove optimality of the spectral gap by following the proof of Proposition \ref{prop:spectralGapSharp} and using the $\breve{\cdot}$ counterparts of the quantities. We get the following result.
\begin{proposition}
	Let $g>\breve{g}$. There exists $\delta>0$ and $x \in \mathcal{W}$ such that
	\[
		\norm{T_t(\breve{x})}_2^2 > \me^{-2tg} \norm{\breve{x}}_2^2, \quad \forall t \in (0, \delta).
	\]
\end{proposition}
Eventually, following again Theorem \ref{thm:spectralGap}, we arrive to the following result.

\begin{theorem} \label{thm:spectralGapKMS}
Let $\mathcal{T}$ be a gaussian QMS with a unique invariant faithful gaussian state and $\zeta=0$.
If $\ker Z^\sharp \breve{S} + \breve{S} Z \neq \{0\}$ then $\mathcal{T}$ has the spectral gap $\breve{g}$.
\end{theorem}
As foreshadowed in the introduction to this section, the only difference between Theorems \ref{thm:spectralGap}
and \ref{thm:spectralGapKMS} is the need to replace $\mathbf{\tilde{S}}$ with $\breve{S}$. Even the condition on
$\ker Z^\sharp \breve{S} + \breve{S} Z$ is the translation of the condition $\CZ>0$ for the GNS embedding since
$-\CZ = \mathbf{Z^* \tilde{S}} + \mathbf{\tilde{S} Z}$.
 However the explicit conditions we had in Theorem \ref{thm:spectralGap} are now replaced by the implicit requirement of a unique invariant faithful gaussian state. This is due to the fact that the sufficient conditions we had in the GNS case, namely $\CZ >0$ and $Z$ stable, are no more necessary for the existence of a spectral gap (see the last comment of subsection \ref{subsect:KMSgap} below).

Another difference of note, is the lack of a counterpart of Proposition \ref{prop:omega0Properties} in the KMS embedding.
Indeed we lack a simple formula to connect the matrix $\breve{S}$ to the operators $Z$ and $C$, that describe the QMS.
For this very reason it is also difficult to compare the spectral gaps that we obtain
with respect to the two different embeddings. We believe that  $\breve{g}$ is bigger than $g$, however we do not have a definite proof.
The next example support this claim as the prototypical example of a gaussian QMS.

\subsection{one dimensional case: KMS embedding} \label{subsect:KMSgap}
In Section \ref{sec:oneDimCase}, we computed the spectral gap for the GNS embedding for a typical model with $d=1$.
Here, we consider the same gaussian QMS and we analyse the situation for the KMS embedding. Starting from the expression for $\mathbf{S}$ in \eqref{eq:Sonedim} we may symplectically diagonalize it writing $\mathbf{S} = \mathbf{M}^* \mathbf{D_\sigma} \mathbf{M}$ with
\[
	\mathbf{M} = \begin{bmatrix}
		\sqrt{\gamma^2 + \Omega^2} & 0 \\
		\frac{\kappa \gamma}{\left(\gamma^2 + \Omega^2 - \kappa^2 \right)\sqrt{\gamma^2 + \Omega^2}} & \frac{1}{\sqrt{\gamma^2 + \Omega^2}}
	\end{bmatrix}, \quad \mathbf{D_\sigma} = \begin{bmatrix}
		\sigma & 0 \\ 0 & \sigma
\end{bmatrix}, \quad  \sigma = \frac{\mu^2 + \lambda^2}{2\gamma} \sqrt{\frac{\gamma^2 + \Omega^2}{\gamma^2 + \Omega^2 - \kappa^2}}.
\]
Since $d=1$, both the matrix $\mathbf{D_\sigma}$ in the previous equation and $\mathbf{D_\nu}$ coming from Lemma \ref{lem:normSpectralGapKMS} are actually multiple of the identity therefore we can write
	\begin{equation} \label{eq:Sbreveonedim}
		\mathbf{\breve{S}} = \frac{\operatorname{csch}\coth^{-1} (\sigma)}{\sigma}\mathbf{S} = \frac{\operatorname{csch}\coth^{-1} (\sigma)}{\sigma} \mathbf{M}^* \left(\sigma \identity\right) \mathbf{M}.
	\end{equation}
	In particular, contrary to what happens in the GNS embedding with $\mathbf{\tilde{S}}$, the operator $\mathbf{\breve{S}}$ does not involve any operations that require the complexification for its definition. Therefore equation \eqref{eq:Sbreveonedim} can simply be formulated using real linear operators, i.e.
	\[
		\breve{S} = \frac{\operatorname{csch}\coth^{-1} (\sigma)}{\sigma}S.
	\]
	This allows us to also avoid consider complexification of the operators $Z,C$ so that the computations for the spectral gap can proceed, using just real linear operators, as follows
	\begin{align*}
		\breve{S}^\frac{1}{2}Z\breve{S}^{-\frac{1}{2}} + \breve{S}^{-\frac{1}{2}}Z^\sharp\breve{S}^\frac{1}{2} &= S^\frac{1}{2} Z S^{-\frac{1}{2}} + S^{-\frac{1}{2}} Z^\sharp S^{\frac{1}{2}} \\
		&= -S^{-\frac{1}{2}} C S^{-\frac{1}{2}},
	\end{align*}
	which is similar to $-CS^{-1}$. Notice that in the calculations we used once again the algebraic
characterization of the covariance matrix $S$, namely $SZ+Z^\sharp S=-C$.
	The eigenvalues of $CS^{-1}$ satisfy now the equation $r^2 - \tr(CS^{-1})r + \frac{\det C}{\det S} = 0$. Borrowing some of the calculations from Section \ref{sec:oneDimCase} we have
	\[
		\tr(CS^{-1}) = 4\gamma, \quad \det(CS^{-1}) = \frac{4\gamma^2(\gamma^2 + \Omega^2 - \kappa^2)}{\Omega^2 + \gamma^2}
	\]
	In particular $2\breve{g}$ is the smallest eigenvalue of $CS^{-1}$ and so
	\[
		\breve{g} = \gamma\left(1 - \frac{\modulo{\kappa}}{\sqrt{\Omega^2 + \gamma^2}}\right).
	\]
	Comparing this quantity with $g$ obtained in Section \ref{sec:oneDimCase} we have $\breve{g}>g$ if and only if
	\[
		\frac{\modulo{\kappa}}{\sqrt{\Omega^2 + \gamma^2}} < \frac{\modulo{\kappa}(\mu^2 + \lambda^2)}{2\sqrt{\mu^2\lambda^2(\Omega^2 + \gamma^2) + \gamma^2\kappa^2}}
	\]
	which simplifies down to (recalling that $\gamma=(\mu^2-\lambda^2)/2$)
	\[
		\gamma^2(\gamma^2+\Omega^2  - \kappa^2) >0.
	\]
	The inequality is now always satisfied since $\gamma > 0$ and $\Omega^2 + \gamma^2 - \kappa^2 > 0$ are necessary and sufficient conditions for the semigroup to have a gaussian invariant state (see \cite{AFP}).

In particular we can set $\lambda=0$, resulting in the non-invertibility of $\CZ$, and still have $\breve{g}>0$. This shows that the conditions for the existence of a unique faithful invariant gaussian state of Theorem \ref{thm:spectralGap}, namely $Z$ stable and $\CZ>0$, are only sufficient for the existence of a spectral gap in the KMS embedding, unlike in the GNS one.

\textbf{Acknowledgements}
The authors are members of GNAMPA-INdAM.
FF  acknowledges the support of the MUR grant ``Dipartimento di Eccellenza 2023--2027'' of Dipartimento di Matematica, Politecnico di Milano
and ``Centro Nazionale di ricerca in HPC, Big Data and Quantum Computing''.
DP, ES and VU have been supported by the MUR grant ``Dipartimento di Eccellenza 2023–2027'' of Dipartimento di Matematica, Universit\`a di Genova.

\appendix
\section{Calculations on $\CZ$}\label{sect:appA}
\begin{lemma} Matrices $\mathbf{Z}$ and $\CZ$ are given by \eqref{eq:mathbfZ} and \eqref{eq:mathbfCZ}.
%	\begin{align*}
%		\mathbf{Z} &= \frac{1}{2}\begin{bmatrix}
%			\Re \left(\left( U - \conj{V}\right)^*\left(U +\conj{V}\right)\right) & \Im\left(\left( U - \conj{V}\right)^*\left(U
% -\conj{V}\right)\right) \\
%			- \Im \left(\left( U + \conj{V}\right)^*\left(U +\conj{V}\right)\right) &
% \Re \left(\left( U + \conj{V}\right)^*\left(U -\conj{V}\right)\right)
%		\end{bmatrix} +
% \begin{bmatrix}
%			-\Im \left(\Omega + \kappa\right) & \Re\left(\kappa - \Omega \right) \\
%			\Re\left(\Omega + \kappa\right) & \Im\left(\kappa - \Omega\right)
%		\end{bmatrix} \\
%		\CZ &= \begin{bmatrix}
%		\left( U + \conj{V} \right)^* \left(U + \conj{V}\right) & -\mi\left( U + \conj{V}\right)^* \left( U - \conj{V} \right) \\
%		\mi \left( U - \conj{V}\right)^* \left(U +  \conj{V}\right) & \left( U -\conj{V} \right)^* \left( U - \conj{V}\right)
%\end{bmatrix}
%	\end{align*}
\end{lemma}
\begin{proof}
	Using the formula of the complexification of a real linear operator we have
	\begin{align*}
		\mathbf{Z} &=  \frac{1}{2} \begin{bmatrix}
			\Re (U^T\conj{U} - V^T\conj{V} + U^TV - V^TU) & \Im (-U^T\conj{U} + V^T\conj{V} + U^TV - V^TU) \\
			\Im (U^T\conj{U} - V^T\conj{V} + U^TV - V^TU) & \Re (U^T\conj{U} - V^T\conj{V} - U^TV + V^TU)
			\end{bmatrix} \\
			&+ \begin{bmatrix}
			-\Im \left(\Omega + \kappa\right) & \Re\left(\kappa - \Omega \right) \\
			\Re\left(\Omega + \kappa\right) & \Im\left(\kappa - \Omega\right)
		\end{bmatrix} \\
		&= \frac{1}{2}\begin{bmatrix}
			\Re \left(\left( U - \conj{V}\right)^*\left(U +\conj{V}\right)\right) & \Im\left(\left( U - \conj{V}\right)^*\left(U -\conj{V}\right)\right) \\
			- \Im \left(\left( U + \conj{V}\right)^*\left(U +\conj{V}\right)\right) & \Re \left(\left( U + \conj{V}\right)^*\left(U -\conj{V}\right)\right)
		\end{bmatrix}  + \begin{bmatrix}
			-\Im \left(\Omega + \kappa\right) & \Re\left(\kappa - \Omega \right) \\
			\Re\left(\Omega + \kappa\right) & \Im\left(\kappa - \Omega\right)
		\end{bmatrix}.
	\end{align*}
	Similarly
	\begin{align*}
		\mathbf{C} &= \begin{bmatrix}
			\Re (\conj{U^*U + V^*V} + U^TV + V^TU) & -\Im (\conj{U^*U + V^*V} - U^TV - V^TU) \\
			\Im (\conj{U^*U + V^*V} + U^TV + V^TU) & \Re (\conj{U^*U + V^*V} - U^TV - V^TU)
		\end{bmatrix}\\
		&= \begin{bmatrix}
			\Re \left( \left(U + \conj{V}\right)^*\left(U + \conj{V}\right)\right) & \Im \left( \left(U + \conj{V}\right)^*\left(U - \conj{V}\right)\right)\\
			-\Im \left( \left(U - \conj{V}\right)^*\left(U + \conj{V}\right)\right) & \Re \left( \left(U - \conj{V}\right)^*\left(U - \conj{V}\right)\right)
		\end{bmatrix}.
	\end{align*}
	Moreover
	\[
		\mathbf{Z}^*\mathbf{J} + \mathbf{JZ} = \mathbf{JZ} - (\mathbf{JZ})^*
	\]
	and
	\begin{align*}
		\mathbf{JZ} &= \frac{1}{2}\begin{bmatrix}
			- \Im \left(\left( U + \conj{V}\right)^*\left(U +\conj{V}\right)\right) & \Re \left(\left( U + \conj{V}\right)^*\left(U -\conj{V}\right)\right) \\
			-\Re \left(\left( U - \conj{V}\right)^*\left(U +\conj{V}\right)\right) & -\Im\left(\left( U - \conj{V}\right)^*\left(U -\conj{V}\right)\right)
		\end{bmatrix} \\
		&+ \begin{bmatrix}
			\Re(\Omega + \kappa) & \Im (\kappa - \Omega) \\
			\Im (\kappa + \Omega) & \Re (\Omega - \kappa)
		\end{bmatrix}
	\end{align*}
	Now note that the first matrix in the expression for $\mathbf{JZ}$ is anti-selfadjoint while the second one is selfadjoint. In particular then
	\[
		\mathbf{Z}^*\mathbf{J} + \mathbf{JZ} = \begin{bmatrix}
			-\Im \left( \left(U + \conj{V}\right)^*\left(U + \conj{V}\right)\right) & \Re \left( \left(U + \conj{V}\right)^*\left(U - \conj{V}\right)\right)\\
			-\Re \left( \left(U - \conj{V}\right)^*\left(U + \conj{V}\right)\right) & -\Im \left( \left(U - \conj{V}\right)^*\left(U - \conj{V}\right)\right)
		\end{bmatrix}.
	\]
	Using the definition of $\CZ = \mathbf{C} -\mi \left( \mathbf{Z}^*\mathbf{J} + \mathbf{JZ} \right)$ we get the desired result.
\end{proof}

\section{Proof of Proposition \ref{prop:gapImpliesCZinvertible}}
This appendix is devoted to the proof of Proposition \ref{prop:gapImpliesCZinvertible}
Let $(z_p,z_q) \in \ker \CZ$ with $z_p, z_q \in \mathbb{C}^d$ and consider
$x =\sum_{j=1}^d(z_{p,j} p_j - z_{q,j} q_j) \rho^{1/2} $. Note that
\begin{align*}
		\sum_{j=1}^d \left(z_{p,j} p_j - z_{q,j} q_j \right) &= \sum_{j=1}^d\left( \Re z_{p,j} p_j - \Re z_{q,j} q_j \right) +\mi \sum_{j=1}^d\left( \Im z_{p,j} p_j - \Im z_{q,j}q_j \right) \\
		&= \sqrt{2} \left( p(z_1) + \mi p(z_2) \right)
	\end{align*}
where $\sqrt{2}z_1 = \Re z_p + \mi \Re z_q$, $\sqrt{2} z_2 = \Im z_p + \mi \Im z_q$.
Which motivates the introduction of the following notation
\[
	\tilde{W}(z_p, z_q) = W(z_1) + \mi W(z_2)
\]
and the ratios of increments
\begin{equation} \label{eq:ratios}
	R_z(r) = \frac{W(rz) - \identity}{-\mi r}, \quad \tilde{R}_{z_p, z_q} (r) = \frac{\tilde{W}(rz_p, rz_q) - (1+\mi)\identity}{-\mi r} = R_{z_1}(r) + \mi R_{z_2}(r).
\end{equation}
In particular it holds
\[	
	\sum_{j=1}^d \left(z_{p,j} p_j - z_{q,j} q_j \right) = \lim_{r \to 0^+} \frac{ W(r z_1) + \mi W( r z_2) - (1 + \mi ) \identity}{-\mi r} = \lim_{r \to 0^+} \tilde{R}_{z_p,z_q}(r)
\]
where the limits holds in the strong operator topology and on a suitable dense domain.
We can now prove that the existence of a spectral gap implies $\CZ > 0$, starting from a lemma.
\begin{lemma} \label{lem:derivative0DoppioWeyl}
	Let $\mathcal{T}$ be a gaussian QMS with an invariant state $\rho$. For every $z,w \in \mathbb{C}^d$ we have
	\[
		\left.\frac{\md}{\md t} \tr \left( \rho \mathcal{T}_t (R_z(r))^* \mathcal{T}_t (
	R_w(r)) \right)\right|_{t=0} = -\scalar{\begin{pmatrix}
			\Re z \\ \Im z
		\end{pmatrix}}{ \mathbf{C_Z} \begin{pmatrix}
			\Re w \\ \Im w
		\end{pmatrix}} \tr \left( \rho W(-rz) W(rw) \right)
	\]
\end{lemma}

\begin{proof}
	Note at first that, from \eqref{eq:explWeyl},
	\begin{align*}
		\mathcal{T}_t (W(z)) \mathcal{T}_t (W(w)) &= \alpha_t(z)\alpha_t(w) W(\me^{tZ}z) W(\me^{tZ}w) \\
		&= \alpha_t(z) \alpha_t(w) \me^{-\mi \Im\scalar{\me^{tZ}z}{\me^{tZ}w}} W(\me^{tZ}(z+w)) \\
		&= \exp\left\{ \int_0^t \rescalar{\me^{sZ}z}{C\me^{sZ}w}\md s- \mi \Im \scalar{\me^{tZ}z}{\me^{tZ}w}\right\} \mathcal{T}_t( W(z + w)),
	\end{align*}
where we set
	$$\alpha_t(z)=\exp\left(-\frac{1}{2}\int_0^t \Re\scalar{\mathrm{e}^{sZ}z}{
  C \mathrm{e}^{sZ}z}\mathrm{d}s\right).$$
%	Therefore
%	\[
%		\tr \left( \rho \mathcal{T}_t (W(z)) \mathcal{T}_t (W(w)) \right) = \exp\left\{ \int_0^t \rescalar{\me^{sZ}z}{C \me^{sZ}w}\md s- \mi \Im \scalar{\me^{tZ}z}{\me^{tZ}w}\right\} \tr \left( \rho W(z + w) \right),
%	\]
	Tracing both sides of the previous equation against $\rho$ and taking the derivative at $t=0$ one gets
	\begin{align*}
		\left.\frac{\md}{\md t} \tr \left( \rho \mathcal{T}_t (W(z)) \mathcal{T}_t (W(w)) \right)\right|_{t=0} &= \left(\rescalar{z}{Cw}- \mi \Im \scalar{Zz}{w} -\mi \Im \scalar{z}{Zw}\right)\\
		&\cdot \me^{-\mi \Im\scalar{z}{w}} \tr \left( \rho W(z + w) \right) \\
		& = \scalar{\begin{pmatrix}
			\Re z \\ \Im z
		\end{pmatrix}}{\mathbf{C_Z}\begin{pmatrix}
			\Re w \\ \Im w
		\end{pmatrix}} \tr \left( \rho W(z) W(w) \right).
	\end{align*}
	where the first equality is due to explicit computations and the invariant property for $\rho$, while the second one is due to \eqref{eq:WeylCCR} and \eqref{legame-prod-scal}. Eventually
	\begin{align*}
		\tr \left( \rho \mathcal{T}_t (R_z(r))^* \mathcal{T}_t (
	R_w(r))\right) &= \frac{1}{r^2}\tr \left( \rho \mathcal{T}_t (W(-rz)-\identity) \mathcal{T}_t (W(rw) - \identity ) \right) \\
	&= \frac{1}{r^2} \tr \left( \rho \mathcal{T}_t (W(-rz)) \mathcal{T}_t (W(rw))\right) \\
	&+\frac{1}{r^2} \left( -\tr (\rho W(-rz)) - \tr (\rho W(rw) ) + \tr(\rho) \right)
	\end{align*}
	Therefore, when taking the derivative at $t=0$ only the first summand does not vanish and the result is proven.
\end{proof}

\begin{lemma} \label{lem:charFunctionPolynomials}
	Let $\mathcal{T}$ be a gaussian QMS with a faithful gaussian invariant state $\rho$. Then for every $z, w \in \mathbb{C}^d$ it holds
	\[
		\lim_{r \to 0^+}\norm{\tilde{R}_{z_p,z_q}(r)\rho^{1/2} }_2^2 >0
	\]
\end{lemma}

\begin{proof}
Note at first that, for $z, w \in \mathbb{C}^d$
\begin{align*}
	\tr( \rho R_z(r)^* R_w(r)) &= \frac{1}{r^2}\tr\left(\rho (W(-rz) - \identity)(W(rw) - \identity)\right) \\
	&=\frac{\me^{\mi r^2\Im\scalar{z}{w}}\hat{\rho}(-r(z-w)) - 1}{r^2} - \frac{\hat{\rho}(-rz) - 1}{r^2} - \frac{\hat{\rho}(rw) - 1}{r^2}
\end{align*}
and the limit as $r \to 0^+$ exists and is finite since the characteristic function of $\rho$ is differentiable. In particular it holds
\begin{align*}
	\lim_{r \to 0^+} \tr( \rho R_z(r)^* R_w(r) ) &= \left( \mi \Im\scalar{z}{w} - \frac{1}{2}\Re\scalar{(z-w)}{S(z-w)}\right) \\
	&+ \frac{1}{2}\Re\scalar{z}{Sz} + \frac{1}{2}\Re\scalar{w}{Sw} \\
	&= \rescalar{z}{Sw} + \mi \Im\scalar{z}{w} = \scalar{\begin{bmatrix}
	 \Re z \\ \Im z
	\end{bmatrix}}{ (\mathbf{S} + \mi \mathbf{J}) \begin{bmatrix}
		\Re w \\ \Im w
	\end{bmatrix}}
\end{align*}
thanks to \eqref{legame-prod-scal}.
To conclude that the desired limit exists and is finite it sufficient to note that, using definition \eqref{eq:ratios},
\begin{align*}
	\norm{\tilde{R}_{z_p,z_q}(r) \rho^{1/2} }_2^2 &= \tr \left( \rho \left(\tilde{R}_{z_p,z_q}(r)\right)^*\tilde{R}_{z_p,z_q}(r) \right)  \\
	&= \tr\left( \rho R_{z_1}(r)^* R_{z_1}(r)\right) + \tr\left( \rho R_{z_2}(r)^* R_{z_2}(r)\right) \\
	& +\mi\, \tr\left( \rho R_{z_1}(r)^* R_{z_2}(r)\right) - \mi\, \tr\left( \rho R_{z_2}(r)^* R_{z_1}(r)\right)
\end{align*}
and taking the limit as $r \to 0^+$
\begin{align*}
	\lim_{r \to 0^+} \norm{\tilde{R}_{z_p,z_q}(r)\rho^{1/2}  }_2^2 &= \scalar{ \begin{bmatrix}
		\Re z_1 + \mi \Re z_2 \\ \Im z_1 + \mi \Im z_2
	\end{bmatrix}}{(\mathbf{S} + \mi \mathbf{J}) \begin{bmatrix}
		\Re z_1 + \mi \Re z_2 \\ \Im z_1 + \mi \Im z_2
	\end{bmatrix}} \\
	&= \frac{1}{2}\scalar{ \begin{bmatrix}
		z_p \\ z_q
	\end{bmatrix}}{(\mathbf{S} + \mi \mathbf{J}) \begin{bmatrix}
		z_p \\ z_q
	\end{bmatrix}} >0
\end{align*}
Where positivity of the last quantity follows from the fact that $\mathbf{S} + \mi \mathbf{J} = \conj{\mathbf{S} - \mi \mathbf{J}} >0$.
\end{proof}

\smallskip
We are now ready to prove Proposition \ref{prop:gapImpliesCZinvertible}.

\begin{proof}[Proof of Proposition \ref{prop:gapImpliesCZinvertible}]
	Let $[z_p , z_q]^T$ be an eigenvector for $\CZ$ associated with the eigenvalue $0$. For every $g,r >0$ consider the function
	\[
		f_{g,r}(t) = \norm{\mathcal{T}_t \left(\tilde{R}_{z_p,z_q}(r) \right) \rho^{1/2} }_2^2 - \me^{-2gt} \norm{\tilde{R}_{z_p,z_q}(r) \rho^{1/2}}_2^2
	\]
	Clearly $f_{g,r}(0) = 0$ and we will show that for every $g>0$ there exists $r_g > 0$ such that $f^\prime_{g,r_g}(0) >0$. In this way, at least for $t>0$ small enough, $f_{g,r_g}(t) > 0$ contradicting the existence of a spectral gap.
	The derivative at $t=0$ of $f_{g,r}(t)$ is
	\[
		f^\prime_{g,r}(0) = \left. \frac{\md}{\md t} \norm{\mathcal{T}_t \left(\tilde{R}_{z_p,z_q} (r)\right) \rho^{1/2}}_2^2\right|_{t=0} +2g \norm{\tilde{R}_{z_p,z_q}(r)\rho^{1/2} }_2^2.
	\]
	The first summand can be further expanded using
	\begin{align*}
		\norm{\mathcal{T}_t \left(\tilde{R}_{z_p,z_q}(r) \right) \rho^{1/2}}_2^2 & =\tr \left( \rho \mathcal{T}_t (\tilde{R}_{z_p,z_q}(r))^* \mathcal{T}_t (\tilde{R}_{z_p,z_q}(r))\right) \\
		&= \tr \left[ \rho \mathcal{T}_t (R_{z_1}(r))^*\mathcal{T}_t (R_{z_1}(r)))\right] + \tr \left[ \rho \mathcal{T}_t (R_{z_2}(r))^*\mathcal{T}_t (R_{z_2}(r))\right] \\
		& +\mi \tr \left[ \rho \mathcal{T}_t (R_{z_1}(r))^*\mathcal{T}_t (R_{z_2}(r))\right] - \mi \tr \left[ \rho \mathcal{T}_t (R_{z_2}(r))^*\mathcal{T}_t (R_{z_1}(r))\right]
	\end{align*}
	and then taking the derivative at $t=0$, using Lemma \ref{lem:derivative0DoppioWeyl}
	\begin{align*}
		\left.\frac{\md}{\md t} \norm{\mathcal{T}_t \left(\tilde{R}_{z_p,z_q}(r) \right) \rho^{1/2} }_2^2 \right|_{t=0} & =  - \scalar{\begin{pmatrix}
			\Re z_1 \\ \Im z_1
		\end{pmatrix}}{\mathbf{C_Z}\begin{pmatrix}
			\Re z_1 \\ \Im z_1
		\end{pmatrix}} \tr \left( \rho W(-rz_1) W(rz_1) \right)\\
		&- \scalar{\begin{pmatrix}
			\Re z_2 \\ \Im z_2
		\end{pmatrix}}{\mathbf{C_Z} \begin{pmatrix}
			\Re z_2 \\ \Im z_2
		\end{pmatrix}} \tr \left( \rho W(-rz_2) W(rz_2) \right) \\
		& + \mi \scalar{\begin{pmatrix}
			\Re z_2 \\ \Im z_2
		\end{pmatrix}}{\mathbf{C_Z}\begin{pmatrix}
			\Re z_1 \\ \Im z_1
		\end{pmatrix}} \tr \left(\rho W(-rz_2) W(rz_1) \right) \\
		& -\mi \scalar{\begin{pmatrix}
			\Re z_1 \\ \Im z_1
		\end{pmatrix}}{\mathbf{C_Z}\begin{pmatrix}
			\Re z_2 \\ \Im z_2
		\end{pmatrix}} \tr \left( \rho W(-rz_1) W(rz_2) \right).
	\end{align*}
	Letting $r \to 0^+$ we get
	\[
		\lim_{r \to 0^+} \left.\frac{\md}{\md t} \norm{ \mathcal{T}_t \left(\tilde{R}_{z_p,z_q}(r)\right) \rho^{1/2} }_2^2 \right|_{t=0} = \scalar{\begin{pmatrix}
			z_p \\ z_q
		\end{pmatrix}}{\CZ \begin{pmatrix}
			z_p \\ z_q
		\end{pmatrix}} = 0.
	\]
	Eventually, using Lemma \ref{lem:charFunctionPolynomials}, we get
	\begin{align*}
		\lim_{r \to 0^+} f^\prime_{g,r}(0) &= \lim_{r \to 0^+}\norm{\tilde{R}_{z_p,z_q}(r)\rho^{1/2}  }_2^2 > 0 .
	\end{align*}
\end{proof}

\section{Proof of Lemma \ref{lem:normSpectralGapKMS}} \label{sec:appC}
This appendix contains the proof of Lemma \ref{lem:normSpectralGapKMS}. The first part revolves around the computation of the quantity
\[
	\tr (\rho^\frac{1}{2}W(z)\rho^\frac{1}{2}W(w)),
\]
where $z,w \in \mathbb{C}^d$ and $\rho = \rho_{(0,S)}$ is a faithful gaussian state.
The starting point is to simplify the expression for $\rho$ as one does for gaussian vectors in classical probability. Let us start by introducing some notation.
For $c \in \mathbb{R}^d$, we denote with $D_c$ the real linear operator defined by
\[
	D_c z = \operatorname{diag}(c_1, \dots, c_d) z, \qquad \mathbf{D_c} \begin{bmatrix}
	x \\ y
	\end{bmatrix} = \operatorname{diag}(c_1, \dots, c_d, c_1, \dots, c_d) \begin{bmatrix}
	x \\ y
	\end{bmatrix}
\]
From \cite{KRPSymm} every faithful gaussian state $\rho = \rho_{(0,S)}$ has the form
\[
	\rho = \Gamma(M)^{-1} \prod_j (1- \me^{-s_j}) \me^{-\sum_j s_j a_j^\dagger a_j} \Gamma (M),
\]
where $M$ is a symplectic, or Bogoliubov,  transformation (i.e. $M^T JM = J$) such that ${M}^TD_\sigma{M} = S$, having set $\sigma \in \mathbb{R}^{d}$ and $\sigma_j = \coth (s_j/2)$. Eventually $\Gamma (M)$ is a unitary operator satisfying
\[
	\Gamma (M) W(z) \Gamma (M)^{-1} = W(Mz).
\]
The previous discussion allows us to have an expression for $\rho^\frac{1}{2}$ that one can work with, explicitly
\begin{equation} \label{eq:rhoCanonicalDecomposition}
	\rho^{-\frac{1}{2}} = \Gamma(M)^{-1} \prod_j (1- \me^{-s_j})^{\frac{1}{2}} \me^{-\sum_j s_j/2 a_j^\dagger a_j} \Gamma (M).
\end{equation}
Before moving to the actual calculations let us recall first \emph{exponential vectors} in $\mathsf{h}$. For any  $f \in \mathbb{C}^d$ we have
\[
	e(f) = \me^{\frac{\modulo{z}^2}{2}}W(f)e(0, \dots, 0) = \sum_{n \geq 0}\sum_{\alpha \in \mathbb{N}^d, \modulo{\alpha} = n} \frac{f_1^{\alpha_1} \cdot \dots \cdot f_d^{\alpha_d}}{\alpha_1! \dots \alpha_d!} e(\alpha_1, \dots, \alpha_d).
\]
In particular for any $z \in \mathbb{C}^d$ it holds
\[
	W(z)e(f) = \exp\left\{ -\frac{\modulo{z}^2}{2} -\scalar{z}{f}\right\} e(z+f).
\]
Exponential vectors are a total set in $\mathsf{h}$ and provide an alternative way to computing traces, as highlighted in the following Lemma, whose proof can be found in \cite{KRPcosa}.
\begin{lemma} \label{lem:identityDecomposition}
For $z = x+ \mi y$
\[
	\identity = \frac{1}{\pi^d} \int_{\mathbb{R}^{2d}} \me^{-\modulo{z}^2} \outerp{e(z)}{e(z)} \md x \md y,
\]
also
\[
	\tr(\rho) = \frac{1}{\pi^d} \int_{\mathbb{R}^{2d}} \me^{-\modulo{z}^2} \scalare{e(z)}{\rho e(z)} \md x \md y
\]
\end{lemma}

We can now start the computations with the following two Lemmas.
\begin{lemma} \label{lem:etothenumber}
	For any $f,g \in \mathbb{C}^d$ and $c_1, \dots, c_d \in (0, \infty)$ it holds
	\[
		\scalar{ e(f)}{ \me^{-\sum_j c_j a_j^\dagger a_j} e(g)} = \exp \left\{ \scalar{f}{\me^{-D_c} g}\right\}.
	\]	
\end{lemma}
\begin{proof}
	Note at first that, for some fixed index $j$,
	\begin{align*}
		\me^{-c_ja_j^\dagger a_j} e(f) &= \sum_{m \geq 0} \sum_{n \geq 0} \sum_{|\alpha| = n} \frac{f_1^{\alpha_1} \cdot \dots \cdot f_d^{\alpha_d}}{\alpha_1! \dots \alpha_d!} \frac{1}{m!} (-c_ja_j^\dagger a_j)^m e(\alpha_1, \dots, \alpha_d) \\
		&= \sum_{m \geq 0}\sum_{n \geq 0} \sum_{|\alpha| = n} \frac{f_1^{\alpha_1} \cdot \dots \cdot f_d^{\alpha_d}}{\alpha_1! \dots \alpha_d!} \frac{1}{m!} (-c_j \alpha_j)^m e(\alpha_1, \dots, \alpha_d) \\
		&=\sum_{n \geq 0} \sum_{|\alpha| = n} \frac{f_1^{\alpha_1} \cdot \dots \cdot f_d^{\alpha_d}}{\alpha_1! \dots \alpha_d!} \me^{-c_j \alpha_j} e(\alpha_1, \dots, \alpha_d) \\
		&= \sum_{n \geq 0} \sum_{|\alpha| = n} \frac{f_1^{\alpha_1} \cdot \dots \cdot (\me^{-c_j} f_j)^{\alpha_j} \cdot  \dots \cdot f_d^{\alpha_d}}{\alpha_1! \dots \alpha_d!} e(\alpha_1, \dots, \alpha_d) \\
		&= e ( \operatorname{diag}(1, \dots,1,  \me^{-c_j},1, \dots, 1) f ).
	\end{align*}
	In this way
	\begin{align*}
		\scalar{ e(f)}{ \me^{\sum_j -c_j a_j^\dagger a_j} e(g)} &= \scalar{e(f)}{\prod_j \left(\me^{-c_ja_j^\dagger a_j}\right)e(g)}\\
		 &= \scalar{e(f)}{e(\me^{-D_c}g)} = \exp\left\{\scalar{f}{\me^{-D_c} g}\right\}.
	\end{align*}
\end{proof}

\begin{lemma} \label{lem:rhoWithWeyl}
	Let $f,g, z, w \in \mathbb{C}^d$ and $c_1, \dots, c_d \in (0, +\infty)$ it holds
	\begin{align*}
		\scalar{e(f)}{W(z)\me^{-\sum_j c_j a_j^\dagger a_j} W(w) e(g)} &= \exp\left\{-\frac{\modulo{z}^2 + \modulo{w}^2}{2} - \scalar{z}{\me^{-D_c}w} +\scalar{f}{\me^{-D_c}g}\right\}\\
		&\exp\left\{\scalar{f}{z + \me^{-D_c}w} - \scalar{w + \me^{-D_c}z}{g}\right\}
	\end{align*}
\end{lemma}
\begin{proof}
	We have
	\begin{align*}
		\scalar{e(f)}{W(z)\me^{-\sum_j c_j a_j^\dagger a_j} W(w) e(g)} &= \scalar{W(-z)e(f)}{\me^{-\sum_j c_j a_j^\dagger a_j} W(w) e(g)} \\
		&= \me^{-\frac{\modulo{z}^2}{2} +\scalar{f}{z}} \me^{-\frac{\modulo{w}^2}{2} -\scalar{w}{g}} \\
		&\cdot \scalar{e(f-z)}{\me^{-\sum_j c_j a_j^\dagger a_j} e(g+w)}.
	\end{align*}
	Now using Lemma \ref{lem:etothenumber} and rearranging the terms we get the desired result.
\end{proof}

We are now ready to prove the formula
\begin{proposition}
	For $z,w \in \mathbb{C}^d$ it holds
	\begin{align*}
		\tr(\rho^{1/2} W(z) \rho^{1/2} W(w)) &= \exp \left\{-\frac{1}{2} \left(\Re\scalar{z}{S z} + \Re\scalar{w}{S w} + 2\Re\scalar{z}{\breve{S} w} \right)\right\}
	\end{align*}
	with $\breve{S} = M^TD_\xi M$, $\xi_j = \operatorname{csch}(s_j/2)$ and $M$ coming from \eqref{eq:rhoCanonicalDecomposition}.
\end{proposition}
\begin{proof}
	We start by noticing that using expression \eqref{eq:rhoCanonicalDecomposition}, the commutation rule for Weyl operators, the cyclic property of the trace and Lemma \ref{lem:identityDecomposition} we can write
	\begin{align*}
	\tr &(\rho^\frac{1}{2} W(z) \rho^\frac{1}{2} W(w)) = \prod_j (1-\me^{-s_j}) \tr \left( \me^{-\frac{1}{2}\sum_j s_ja_j^\dagger a_j} W(Mz) \me^{-\frac{1}{2}\sum_j s_ja_j^\dagger a_j} W(Mw)\right) \\
	&= \frac{\prod_j (1-\me^{-s_j})}{\pi^{2d}}\int_{\mathbb{R}^{4d}} \me^{-(\modulo{f}^2 + \modulo{g}^2)}\scalar{e(f)}{W\left( \frac{Mw}{2}\right) \me^{-\frac{1}{2}\sum_j s_ja_j^\dagger a_j}W\left( \frac{Mz}{2}\right) e(g)} \\
	& \qquad\qquad\qquad\qquad\qquad \cdot \scalar{e(g)}{W\left( \frac{Mz}{2}\right) \me^{-\frac{1}{2}\sum_j s_ja_j^\dagger a_j}W\left( \frac{Mw}{2}\right) e(f)}\md f \md g
\end{align*}
Eventually, using Lemma \ref{lem:rhoWithWeyl}, we obtain that
\[
	\tr (\rho^\frac{1}{2} W(z) \rho^\frac{1}{2} W(w))  = \frac{\prod_j (1-\me^{-s_j})}{\pi^{2d}}\int_{\mathbb{R}^{4d}} e^{\alpha_{(Mz,Mw)} (f,g)} \md f \md g,
\]
where
\begin{align*}
	\alpha_{(z,w)} (f,g) &= \exp\left\{ -\frac{\modulo{z}^2 + \modulo{w}^2}{4} - \frac{1}{2} \Re\scalar{z}{\me^{-\frac{1}{2} D_s} w}\right\} \\
	&\exp \left\{-\modulo{f}^2 - \modulo{g}^2 + 2\Re\scalar{f}{\me^{-\frac{1}{2}D_s}g} \right\} \\
	&\cdot \exp\left\{ \mi \Im\scalar{f}{w+\me^{-\frac{1}{2}D_s}z} +\mi \Im\scalar{g}{ z + \me^{-\frac{1}{2}D_s}w}\right\}
\end{align*}
We can rewrite the previous expression in matrix form in the following way
\begin{align*}
	\alpha_{(z,w)} (f,g) &= \exp\left\{
	-\frac{1}{4}
	\scalar{
	\begin{pmatrix}
		\Re z \\ \Im z \\ \Re w \\ \Im w
	\end{pmatrix}
	}{
	\begin{pmatrix}
		\identity & \me^{-\frac{1}{2}\mathbf{D_s}} \\ \me^{-\frac{1}{2}\mathbf{D_s}} & \identity
	\end{pmatrix}
	\begin{pmatrix}
		\Re z \\ \Im z \\ \Re w \\ \Im w
	\end{pmatrix}
	}
	\right\} \\
	& \exp \left\{-
	\scalar{
	\begin{pmatrix}
		\Re f \\ \Im f \\ \Re g \\ \Im g
\end{pmatrix}
}{
\begin{pmatrix}
	\identity & -\me^{-\frac{1}{2}\mathbf{D_s}} \\
	-\me^{-\frac{1}{2}\mathbf{D_s}} & \identity
\end{pmatrix}
\begin{pmatrix}
		\Re f \\ \Im f \\ \Re g \\ \Im g
\end{pmatrix}
}
	\right\} \\
	&\cdot \exp\left\{ \mi
	\scalar{
	\begin{pmatrix}
		\me^{-\frac{1}{2}\mathbf{D_s}}\mathbf{J} & \mathbf{J} \\
		 \mathbf{J} & \me^{-\frac{1}{2}\mathbf{D_s}}\mathbf{J}
	\end{pmatrix}
	\begin{pmatrix}
		\Re z \\ \Im z \\ \Re w \\ \Im w
	\end{pmatrix}
	}{
	\begin{pmatrix}
		\Re f \\ \Im f \\ \Re g \\ \Im g
\end{pmatrix}
}
	\right\}
\end{align*}
Recall now for $x,b \in \mathbb{R}^n$ and $A \in M_n(\mathbb{R})$ we have the formula for the gaussian integral
\[
	\int_{\mathbb{R}^n} \me^{ -\frac{1}{2}\scalar{x}{Ax} + \mi \scalar{b}{x}} \md x = \sqrt{\frac{(2\pi)^n}{\det A}} \exp\left\{ -\frac{1}{2} \scalar{b}{A^{-1}b} \right\}.
\]
We want to use it with
\[
	A = 2\begin{pmatrix}
	\identity & -\me^{-\frac{1}{2}\mathbf{D_s}} \\
	-\me^{-\frac{1}{2}\mathbf{D_s}} & \identity
\end{pmatrix}, \quad b= \begin{pmatrix}
		\me^{-\frac{1}{2}\mathbf{D_s}}\mathbf{J} & \mathbf{J} \\
		 \mathbf{J} & \me^{-\frac{1}{2}\mathbf{D_s}}\mathbf{J}
	\end{pmatrix}
	\begin{pmatrix}
		\Re z \\ \Im z \\ \Re w \\ \Im w
	\end{pmatrix}.
\]
Note that, by rearranging rows and columns of $A$ we get
\[
	\frac{1}{2}A = \begin{pmatrix}
		A_1 & &  & \\
		& A_2 & & \\
		& & \ddots & \\
		&&& A_d
	\end{pmatrix}, \quad A_j = \begin{pmatrix}
		1 & 0 & -\me^{-\frac{s_j}{2}} & 0 \\
		0 & 1 & 0 & -\me^{-\frac{s_j}{2}}  \\
		-\me^{-\frac{s_j}{2}}  & 0 & 1 & 0 \\
		0 & -\me^{-\frac{s_j}{2}}  & 0 & 1
	\end{pmatrix},
\]
and for each block we can easily compute $\det A_j = (1-\me^{-s_j})^2$ while
\[
	A_j^{-1} = \frac{1}{1-\me^{-s_j}}\begin{pmatrix}
		1 & 0 & \me^{-\frac{s_j}{2}} & 0 \\
		0 & 1 & 0 & \me^{-\frac{s_j}{2}}  \\
		\me^{-\frac{s_j}{2}}  & 0 & 1 & 0 \\
		0 & \me^{-\frac{s_j}{2}}  & 0 & 1
	\end{pmatrix}.
\]
In particular $\det A = 2^{4d} \prod_j (1-\me^{-s_j})^2$ and rearranging rows and columns to the initial configuration we obtain
\[
	A^{-1} = \frac{1}{2}\begin{pmatrix}
		\identity & \me^{-\frac{1}{2} \mathbf{D_s}} \\
		\me^{-\frac{1}{2}\mathbf{D_s}} & \identity
	\end{pmatrix} \begin{pmatrix}
		\identity- \me^{-\mathbf{D_s}} & 0 \\ 0 & \identity-\me^{-\mathbf{D_s}}
	\end{pmatrix}^{-1}.
\]
Therefore we have proved
\[
	\tr (\rho^\frac{1}{2} W(z) \rho^\frac{1}{2} W(w)) = \me^{-\frac{1}{2} \beta (Mz, Mw)}
\]
with
\begin{align*}
	\beta (z,w) &= \scalar{ \begin{pmatrix}
		\me^{-\frac{1}{2}\mathbf{D_s}}\mathbf{J} & \mathbf{J} \\
		 \mathbf{J} & \me^{-\frac{1}{2}\mathbf{D_s}}\mathbf{J}
	\end{pmatrix}
	\begin{pmatrix}
		\Re z \\ \Im z \\ \Re w \\ \Im w
	\end{pmatrix} }{ A^{-1}\begin{pmatrix}
		\me^{-\frac{1}{2}\mathbf{D_s}}\mathbf{J} & \mathbf{J} \\
		 \mathbf{J} & \me^{-\frac{1}{2}\mathbf{D_s}}\mathbf{J}
	\end{pmatrix}
	\begin{pmatrix}
		\Re z \\ \Im z \\ \Re w \\ \Im w
	\end{pmatrix}} \\
	&+ \frac{1}{2}\scalar{
	\begin{pmatrix}
		\Re z \\ \Im z \\ \Re w \\ \Im w
	\end{pmatrix}
	}{
	\begin{pmatrix}
		\identity & \me^{-\frac{1}{2}\mathbf{D_s}} \\ \me^{-\frac{1}{2}\mathbf{D_s}} & \identity
	\end{pmatrix}
	\begin{pmatrix}
		\Re z \\ \Im z \\ \Re w \\ \Im w
	\end{pmatrix}
	} \\
	&= \scalar{
	\begin{pmatrix}
		\Re z \\ \Im z \\ \Re w \\ \Im w
	\end{pmatrix}
	}{
	B
	\begin{pmatrix}
		\Re z \\ \Im z \\ \Re w \\ \Im w
	\end{pmatrix}
	} + \frac{1}{2}\scalar{
	\begin{pmatrix}
		\Re z \\ \Im z \\ \Re w \\ \Im w
	\end{pmatrix}
	}{
	\begin{pmatrix}
		\identity & \me^{-\frac{1}{2}\mathbf{D_s}} \\ \me^{-\frac{1}{2}\mathbf{D_s}} & \identity
	\end{pmatrix}
	\begin{pmatrix}
		\Re z \\ \Im z \\ \Re w \\ \Im w
	\end{pmatrix}
	}
\end{align*}
where we set
\[
	B = \begin{pmatrix}
		\me^{-\frac{1}{2}\mathbf{D_s}}\mathbf{J} & \mathbf{J} \\
		 \mathbf{J} & \me^{-\frac{1}{2}\mathbf{D_s}}\mathbf{J}
	\end{pmatrix}^* A^{-1}\begin{pmatrix}
		\me^{-\frac{1}{2}\mathbf{D_s}}\mathbf{J} & \mathbf{J} \\
		 \mathbf{J} & \me^{-\frac{1}{2}\mathbf{D_s}}\mathbf{J}
	\end{pmatrix}.
\]
Now note that every matrix that appears as a block in the previous expression commutes with the others (e.g. $\mathbf{J}$ with $\me^{-\frac{1}{2}\mathbf{D_s}}$ or $\me^{-\frac{1}{2}\mathbf{D_s}}$ with $(\identity - \me^{-\mathbf{D_s}})^{-1}$). This observation, alongside the explicit expression we obtained for $A^{-1}$ leads to
\[
	B = \frac{1}{2}\begin{pmatrix}
		\identity + 3\me^{-\mathbf{D_s}} & \me^{-\frac{1}{2}\mathbf{D_s}}( 3\identity + \me^{-\mathbf{D_s}}) \\
		\me^{-\frac{1}{2}\mathbf{D_s}}( 3\identity + \me^{-\mathbf{D_s}}) & \identity + 3\me^{-\mathbf{D_s}}
\end{pmatrix}\begin{pmatrix}
		\identity- \me^{-\mathbf{D_s}} & 0 \\ 0 & \identity-\me^{-\mathbf{D_s}}
	\end{pmatrix}^{-1}.
\]
Putting it all back together we obtain
\begin{align*}
	\beta (z,w) &= \scalar{
	\begin{pmatrix}
		\Re z \\ \Im z \\ \Re w \\ \Im w
	\end{pmatrix}
	}{
	\begin{pmatrix}
		\identity + \me^{-\mathbf{D_s}} &2 \me^{-\frac{1}{2}\mathbf{D_s}} \\ 2\me^{-\frac{1}{2}\mathbf{D_s}} & \identity + \me^{-\mathbf{D_s}}
	\end{pmatrix}\begin{pmatrix}
		\identity- \me^{-\mathbf{D_s}} & 0 \\ 0 & \identity-\me^{-\mathbf{D_s}}
	\end{pmatrix}^{-1}
	\begin{pmatrix}
		\Re z \\ \Im z \\ \Re w \\ \Im w
	\end{pmatrix}.
	}
\end{align*}
Performing the final calculation we obtain
\[
	\beta (z,w) = \scalar{\begin{pmatrix}
		\Re z \\ \Im z \\ \Re w \\ \Im w
	\end{pmatrix}}{
	\begin{pmatrix}
		\mathbf{D_\sigma} & \mathbf{D_\xi} \\
		\mathbf{D_\xi} & \mathbf{D_\sigma}
	\end{pmatrix}\begin{pmatrix}
		\Re z \\ \Im z \\ \Re w \\ \Im w
	\end{pmatrix}
	} = \Re\scalar{z}{D_\sigma z} + \Re\scalar{w}{D_\sigma w} + 2\Re\scalar{z}{D_\xi w}
\]
where $\xi_j = \operatorname{csch}(s_j/2)$, ending the proof.
\end{proof}

We are now ready to prove Lemma \ref{lem:normSpectralGapKMS} following the same lines of Lemma \ref{lem:normSpectralGap} but using the new formula for the trace.

\begin{proof}[ Proof of Lemma \ref{lem:normSpectralGapKMS}]
	Using \eqref{eq:brevex} we have
	\begin{align*}
	\norm{T_t(\breve{x})}^2 &= \sum_{j,k} \conj{\eta_j}\eta_k\left[ \tr(\rho^\frac{1}{2} \mathcal{T}_t(W(-z_j)) \rho^\frac{1}{2}\mathcal{T}_t(W(z_k))) - \me^{-\frac{1}{2}\Re\scalar{z_j}{Sz_j} - \frac{1}{2}\Re\scalar{z_k}{Sz_k}}\right] \\
	&= \sum_{j,k} \conj{\eta_j}\eta_k\left[ c_t(z_j)c_t(z_k)\tr(\rho^\frac{1}{2} W(-\me^{tZ}z_j) \rho^\frac{1}{2}W(\me^{tZ}z_k)) - \me^{-\frac{1}{2}\Re\scalar{z_j}{Sz_j} - \frac{1}{2}\Re\scalar{z_k}{Sz_k}}\right]
	\end{align*}
where $c_t(z_j),c_t(z_k)$ are defined as in \eqref{eq:ctz}.	Recalling that
	\[
		\Re\scalar{\me^{tZ}z}{S\me^{tZ}z} = \int_t^{\infty} \Re\scalar{\me^{sZ}z}{C\me^{sZ}z} \md s
	\]
	we obtain
	\[
		c_t(z_j)c_t(z_k)\tr(\rho^\frac{1}{2} W(-\me^{tZ}z_j) \rho^\frac{1}{2}W(\me^{tZ}z_k)) = \me^{-\frac{1}{2}\Re\scalar{z_j}{Sz_j} - \frac{1}{2}\Re\scalar{z_k}{Sz_k}} \me^{\Re\scalar{\me^{tZ}z_j}{\breve{S}\me^{tZ}z_k}}.
	\]
	Therefore we get
	\[
		\norm{T_t(\breve{x})}^2 = \sum_{j,k} \conj{\xi_j}\xi_k \left[\me^{\Re\scalar{\me^{tZ}z_j}{\breve{S}\me^{tZ}z_k}} - 1 \right].
	\]
	The remaining part of the Lemma follows immediately.
\end{proof}

{\footnotesize

\bigskip

Authors' addresses:
\begin{itemize}
\item[$^{(1)}$] Mathematics Department, Politecnico di Milano,
Piazza Leonardo da Vinci 32, I - 20133 Milano, Italy \quad {\tt franco.fagnola@polimi.it}
\item[$^{(2)}$] Mathematics Department, University of Genova,
Via Dodecaneso, I - 16146 Genova, Italy
\end{itemize}

\end{document}